\documentclass[12pt,twoside]{article}
\usepackage{amsfonts,amsmath,amssymb}
\usepackage{geometry}

\newtheorem{theorem}{Theorem}[section]

\newtheorem{definition}{Definition}[section]

\newtheorem{lemma}{Lemma}[section]

\newtheorem{remark}[theorem]{Remark}

\numberwithin{equation}{section}
\newenvironment{proof}[1][Proof]{\noindent\textbf{#1.} }{\ \rule{0.5em}{0.5em}}
\input{tcilatex}
\begin{document}

\title{{\small ON GENERALIZATION OF SOME INEQUALITIES OF CHEBYSHEV'S
FUNCTIONAL USING GENERALIZED KATUGAMPOLA FRACTIONAL INTEGRAL}}
\author{{\small Tariq A. Aljaaidi, Deepak B. Pachpatte}}
\date{}
\maketitle

\begin{abstract}
In this paper we obtain a generalization of some integral inequalities
related to Chebyshev`s functional by using a generalized Katugampola
fractional integral.
\end{abstract}

\bigskip


\section{\protect\small \ INTRODUCTION:}

\ \ \ In (1882), Chebyshev \cite{FF} has given the following functional%
\begin{equation}
T\left( \varphi ,\psi \right) :=\frac{1}{b-a}\int_{a}^{b}\varphi \left(
\gamma \right) \psi \left( \gamma \right) d\gamma -\frac{1}{b-a}\left(
\int_{a}^{b}\varphi \left( \gamma \right) d\gamma \right) \left(
\int_{a}^{b}\psi \left( \gamma \right) d\gamma \right)   \label{fun1}
\end{equation}%
and its extention is obtained by Mitrinovic as (see \cite{oo})%
\begin{eqnarray}
T\left( \varphi ,\psi ,g,h\right)  &:&=\int_{a}^{b}h\left( \gamma \right)
d\gamma \int_{a}^{b}\varphi \left( \gamma \right) \psi \left( \gamma \right)
g\left( \gamma \right) d\gamma +\int_{a}^{b}g\left( \gamma \right) d\gamma
\int_{a}^{b}\varphi \left( \gamma \right) \psi \left( \gamma \right) h\left(
\gamma \right) d\gamma   \notag \\
&&-\left( \int_{a}^{b}\varphi \left( \gamma \right) h\left( \gamma \right)
d\gamma \right) \left( \int_{a}^{b}\psi \left( \gamma \right) g\left( \gamma
\right) d\gamma \right)   \label{fun2} \\
&&-\left( \int_{a}^{b}\varphi \left( \gamma \right) g\left( \gamma \right)
dx\right) \left( \int_{a}^{b}\psi \left( \gamma \right) h\left( \gamma
\right) d\gamma \right) ,  \notag
\end{eqnarray}%
where $\varphi $ and $\psi $ are two integrable functions on $\left[ a,b%
\right] .$ Many researchers have given considerable attention to the both
functionals and a number of inequalities and a number of extensions,
generalizations and variants have appeared in the literature, for more
details see (\cite{jj}, \cite{hh}, \cite{ii}, \cite{ll}, \cite{mm}). The
functional (\ref{fun2}) is a Chebyshev extension of the functional (\ref%
{fun1}) If $\varphi $ and $\psi $ satisfies the following condition%
\begin{equation*}
\left( \varphi \left( \tau \right) -\varphi \left( \gamma \right) \right)
\left( \psi \left( \tau \right) -\psi \left( \gamma \right) \right) \geq 0,
\end{equation*}%
for any $\tau ,\gamma \in \left[ a,b\right] ,$ then $\varphi $ and $\psi $
are synchronous on $\left[ a,b\right] $, moreover, $T\left( \varphi ,\psi
,g,h\right) \geq 0$ (see also \cite{oo}). It should be noted that the sign
of this inequality is reversed if $\varphi $ and $\psi $ are asynchronous on 
$\left[ a,b\right] ,$ namely $\left( \varphi \left( \tau \right) -\varphi
\left( \gamma \right) \right) \left( \psi \left( \tau \right) -\psi \left(
\gamma \right) \right) \leq 0,$ for any $\tau ,\gamma \in \left[ a,b\right] .
$

One of the most important things to note in this work is the celebrated
Chebyshev functional \cite{FF}%
\begin{equation}
T\left( \varphi ,\psi ,g\right) :=\int_{a}^{b}g\left( \gamma \right) d\gamma
\int_{a}^{b}\varphi \left( \gamma \right) \psi \left( \gamma \right) g\left(
\gamma \right) d\gamma +\int_{a}^{b}g\left( \gamma \right) \varphi \left(
\gamma \right) d\gamma \int_{a}^{b}\psi \left( \gamma \right) g\left( \gamma
\right) d\gamma  \label{fun3}
\end{equation}%
where $\varphi $ and $\psi $ are two differentiable functions on $\left[ a,b%
\right] ,$ and $p$ is positive and integrable functions on $\left[ a,b\right]
,$ this functional has many applications in probability and statistics and
other fields.

In \cite{pp}, (see also \cite{qq}) Dragomir proved that, if $\varphi $ and $%
\psi $ are two differentiable functions such that $\varphi ^{\prime }\in
L_{s}\left( a,b\right) ,\psi ^{\prime }\in L_{v}\left( a,b\right) ,r>1,\frac{%
1}{s}+\frac{1}{v}=1,$ then%
\begin{equation}
2\left\vert T\left( \varphi ,\psi ,g\right) \right\vert \leq \left\Vert
\varphi ^{\prime }\right\Vert _{s}\left\Vert \psi ^{\prime }\right\Vert _{v} 
\left[ \int_{a}^{b}\left\vert \tau -\gamma \right\vert g\left( \tau \right)
g\left( \gamma \right) d\tau d\gamma \right] .  \label{fun4}
\end{equation}
\ For inequality (\ref{fun4}), Dahmani in \cite{rr}, (see also \cite{ss}, 
\cite{tt}) proved fractional version of the inequality%
\begin{eqnarray}
&&\emph{2}\left\vert I^{\alpha }g\left( x\right) I^{\alpha }\left( g\varphi
\psi \right) \left( x\right) -I^{\alpha }\left( g\varphi \right) \left(
x\right) I^{\alpha }\left( g\psi \right) \left( x\right) \right\vert  \notag
\\
&\leq &\frac{\left\Vert \varphi ^{\prime }\right\Vert _{s}\left\Vert \psi
^{\prime }\right\Vert _{v}}{\Gamma ^{2}\left( \alpha \right) }%
\int_{a}^{b}\int_{a}^{b}\left( x-\tau \right) ^{\alpha -1}\left( x-\gamma
\right) ^{\alpha -1}\left\vert \tau -\gamma \right\vert g\left( \tau \right)
g\left( \gamma \right) d\tau d\gamma  \notag \\
&\leq &\left\Vert \varphi ^{\prime }\right\Vert _{s}\left\Vert \psi ^{\prime
}\right\Vert _{v}x\left( I^{\alpha }g\left( x\right) \right) ^{2},
\label{fun5}
\end{eqnarray}%
for all $\alpha >0,x>0.$

Motivated by (\cite{rr}, \cite{ss}, \cite{tt}), in this paper we establish
some new fractional inequalities for Chebyshev functional involving
generalized Katugampola fractional integral. Our results in this paper are
organized in three sections, first and second sections related to Chebyshev
functional in case of synchronous functions and the third section related to
Chebyshev functional in case of differentiable functions whose derivatives
belong to $L_{p}\left( \left[ 0,\infty \right] \right) $.

\bigskip 

\bigskip 

\section{\protect\small \protect\bigskip PRELIMINARIES:}

\bigskip Now, in this section we give the necessary notation and basic
definitions used in our subsequent discussion. For more details see ( \cite%
{aa}, \cite{BB}, \cite{CC} ).

\begin{definition}
Consider the space $X_{c}^{p}\left( a,b\right) \left( c\in 
\mathbb{R}
,1\leq p\leq \infty \right) $, of those complex valued Lebesgue measurable
functions $\varphi $ on $\left( a,b\right) $ for which the norm $\left\Vert
\varphi \right\Vert _{X_{c}^{p}}<\infty $, such that%
\begin{equation*}
\left\Vert \varphi \right\Vert _{X_{c}^{p}}=\left(
\dint\limits_{x}^{b}\left\vert x^{c}\varphi \right\vert ^{p}\frac{dx}{x}%
\right) ^{\frac{1}{p}}\text{ \ \ \ }\left( 1\leq p<\infty \right)
\end{equation*}%
and \ 
\begin{equation*}
\left\Vert \varphi \right\Vert _{X_{c}^{\infty }}=\sup ess_{x\in \left(
a,b\right) }\left[ x^{c}\left\vert \varphi \right\vert \right] .
\end{equation*}%
In particular, when $c=1/p,$ the space $X_{c}^{p}\left( a,b\right) $
coincides with the space $L^{p}\left( a,b\right) .$
\end{definition}

\begin{definition}
\label{def1} The left and right-sided fractional integrals of a function $%
\varphi $ where $\varphi \in $ $X_{c}^{p}\left( a,b\right) ,$ $\alpha >0,$
and $\beta ,\rho ,\eta ,k\in 
\mathbb{R}
,$ are defined respectively by%
\begin{equation}
^{\rho }\mathcal{I}_{a^{+},\eta ,k}^{\alpha ,\beta }\varphi \left( x\right) =%
\frac{\rho ^{1-\beta }x^{k}}{\Gamma \left( \alpha \right) }%
\dint\limits_{a}^{x}\frac{\tau ^{\rho \left( \eta +1\right) -1}}{\left(
x^{\rho }-\tau ^{\rho }\right) ^{1-\alpha }}\varphi \left( \tau \right)
d\tau ,\text{ \ \ \ \ \ \ \ \ }0\leq a<x<b\leq \infty  \label{oper2}
\end{equation}%
and%
\begin{equation}
^{\rho }\mathcal{I}_{b^{-},\eta ,k}^{\alpha ,\beta }\varphi \left( x\right) =%
\frac{\rho ^{1-\beta }x^{\rho \eta }}{\Gamma \left( \alpha \right) }%
\dint\limits_{x}^{b}\frac{\tau ^{k+\rho -1}}{\left( \tau ^{\rho }-x^{\rho
}\right) ^{1-\alpha }}\varphi \left( \tau \right) d\tau ,\text{ \ \ \ \ \ \
\ \ }0\leq a<x<b\leq \infty ,  \label{eq1}
\end{equation}%
if the integral exist.
\end{definition}

In this paper we will use the left-sided fractional integrals (\ref{oper2}),
to present and discuss our new results, also we consider $a=0,$ in (\ref%
{oper2}), to obtain \ \ \ 
\begin{equation*}
^{\rho }\mathcal{I}_{\eta ,k}^{\alpha ,\beta }\varphi \left( x\right) =\frac{%
\rho ^{1-\beta }x^{k}}{\Gamma \left( \alpha \right) }\dint\limits_{0}^{x}%
\frac{\tau ^{\rho \left( \eta +1\right) -1}}{\left( x^{\rho }-\tau ^{\rho
}\right) ^{1-\alpha }}\varphi \left( \tau \right) d\tau .
\end{equation*}
\ \ \ Now we define the following function as in \cite{CC}: let $x>0,\alpha
>0,\rho ,k,\beta ,\eta \in 
\mathbb{R}
$, then 
\begin{equation*}
\Lambda _{x,k}^{\rho ,\beta }\left( \alpha ,\eta \right) =\frac{\Gamma
\left( \eta +1\right) }{\Gamma \left( \eta +\alpha +1\right) }\rho ^{-\beta
}x^{k+\rho \left( \eta +\alpha \right) }.
\end{equation*}%
Note that the Definition (\ref{def1}) is more generalized and can be reduce
to six cases by change its parameters with convenient choice as follows:

- Liouville fractional integral, if $\eta =0,a=0,k=0,$ and taking the limit $%
\rho \rightarrow 1,$ [\cite{dd}, p. 79].

- Weyl fractional integral, if $\eta =0,a=-\infty ,k=0,$ and taking the
limit $\rho \rightarrow 1,$ [\cite{ee}, p. 50].

-Riemann-Liouville fractional integral, if $\eta =0,k=0,$ and taking the
limit $\rho \rightarrow 1,$[\cite{dd}, p. 69].

-Katugampola fractional integral, if $\beta =\alpha ,k=0,\eta =0,$ [\cite{aa}%
].

-Erd\'{e}lyi-Kober fractional integral, if $\beta =0,k=-\rho \left( \alpha
+\eta \right) ,$ [ \cite{dd}, p.105].

-Hadamard fractional integral, if $\beta =\alpha ,k=0,\eta =0^{+},$ and
taking the limit $\rho \rightarrow 1,$ [\cite{dd}, p. 110].

\section{\protect\small Generalized fractional inequality for Chebyshev's
functional:}

\ \ \ In this section we establish inequality for Chebyshev functional \cite%
{FF}, deals with same parameters

\begin{theorem}
\label{thm1} Let $\varphi $ and $\psi $ be two integrable and synchronous
functions on $\left[ 0,\infty \right) .$ Then for all $x>0,\alpha >0,\rho
>0,k,\beta ,\eta \in 
\mathbb{R}
$ we have:%
\begin{equation}
^{\rho }\mathcal{I}_{\eta ,k}^{\alpha ,\beta }\left( \varphi \psi \right)
\left( x\right) \geq \frac{1}{\Lambda _{x,k}^{\rho ,\beta }\left( \alpha
,\eta \right) }\text{ }^{\rho }\mathcal{I}_{\eta ,k}^{\alpha ,\beta }\varphi
\left( x\right) \text{ }^{\rho }\mathcal{I}_{\eta ,k}^{\alpha ,\beta }\psi
\left( x\right) .  \label{INQ1}
\end{equation}
\end{theorem}

\begin{proof}
For a synchronous functions $\varphi ,\psi $ on $\left[ 0,\infty \right) ,$
we have for all $\tau \geq 0,\gamma \geq 0:$%
\begin{equation*}
\left( \varphi \left( \tau \right) -\varphi \left( \gamma \right) \right)
\left( \psi \left( \tau \right) -\psi \left( \gamma \right) \right) \geq 0.
\end{equation*}%
Therefore%
\begin{equation}
\varphi \left( \tau \right) \psi \left( \tau \right) +\varphi \left( \gamma
\right) \psi \left( \gamma \right) \geq \varphi \left( \tau \right) \psi
\left( \gamma \right) +\varphi \left( \gamma \right) \psi \left( \tau
\right) .  \label{inq1}
\end{equation}
Multiplying both sides of (\ref{inq1}) by $\frac{\rho ^{1-\beta }x^{k}}{%
\Gamma \left( \alpha \right) }\frac{\tau ^{\rho \left( \eta +1\right) -1}}{%
\left( x^{\rho }-\tau ^{\rho }\right) ^{1-\alpha }},$ where $\tau \in \left(
0,x\right) ,$ and integrating the resulting inequality over $\left(
0,x\right) ,$ with respect to the variable $\tau ,$ we obtain:%
\begin{eqnarray}
&&\frac{\rho ^{1-\beta }x^{k}}{\Gamma \left( \alpha \right) }\int_{0}^{x}%
\frac{\tau ^{\rho \left( \eta +1\right) -1}}{\left( x^{\rho }-\tau ^{\rho
}\right) ^{1-\alpha }}\varphi \left( \tau \right) \psi \left( \tau \right)
d\tau  \notag \\
&&+\varphi \left( \gamma \right) \psi \left( \gamma \right) \frac{\rho
^{1-\beta }x^{k}}{\Gamma \left( \alpha \right) }\int_{0}^{x}\frac{\tau
^{\rho \left( \eta +1\right) -1}}{\left( x^{\rho }-\tau ^{\rho }\right)
^{1-\alpha }}d\tau  \notag \\
&\geq &\psi \left( \gamma \right) \frac{\rho ^{1-\beta }x^{k}}{\Gamma \left(
\alpha \right) }\int_{0}^{x}\frac{\tau ^{\rho \left( \eta +1\right) -1}}{%
\left( x^{\rho }-\tau ^{\rho }\right) ^{1-\alpha }}\varphi \left( \tau
\right) d\tau  \label{inq3} \\
&&+\varphi \left( \gamma \right) \frac{\rho ^{1-\beta }x^{k}}{\Gamma \left(
\alpha \right) }\int_{0}^{x}\frac{\tau ^{\rho \left( \eta +1\right) -1}}{%
\left( x^{\rho }-\tau ^{\rho }\right) ^{1-\alpha }}\psi \left( \tau \right)
d\tau .  \notag
\end{eqnarray}%
So we have%
\begin{eqnarray}
&&^{\rho }\mathcal{I}_{\eta ,k}^{\alpha ,\beta }\left( \varphi \psi \right)
\left( x\right) +\Lambda _{x,k}^{\rho ,\beta }\left( \alpha ,\eta \right)
\varphi \left( \gamma \right) \psi \left( \gamma \right)  \notag \\
&\geq &\psi \left( \gamma \right) \text{ }^{\rho }\mathcal{I}_{\eta
,k}^{\alpha ,\beta }\varphi \left( x\right) +\varphi \left( \gamma \right) 
\text{ }^{\rho }\mathcal{I}_{\eta ,k}^{\alpha ,\beta }\psi \left( x\right) .
\label{inq4}
\end{eqnarray}%
Now multiplying both sides of (\ref{inq4}) by $\frac{\rho ^{1-\beta }x^{k}}{%
\Gamma \left( \alpha \right) }\frac{\gamma ^{\rho \left( \eta +1\right) -1}}{%
\left( x^{\rho }-\gamma ^{\rho }\right) ^{1-\alpha }},$ where $\gamma \in
\left( 0,x\right) ,$over $\left( 0,x\right) ,$ Then integrating the
resulting inequality over $\left( 0,x\right) ,$ with respect to the variable 
$\gamma ,$ we get: 
\begin{eqnarray*}
&&^{\rho }\mathcal{I}_{\eta ,k}^{\alpha ,\beta }\left( \varphi \psi \right)
\left( x\right) \frac{\rho ^{1-\beta }x^{k}}{\Gamma \left( \alpha \right) }%
\int_{0}^{x}\frac{\gamma ^{\rho \left( \eta +1\right) -1}}{\left( x^{\rho
}-\gamma ^{\rho }\right) ^{1-\alpha }}d\gamma \\
&&+\Lambda _{x,k}^{\rho ,\beta }\left( \alpha ,\eta \right) \frac{\rho
^{1-\beta }x^{k}}{\Gamma \left( \alpha \right) }\int_{0}^{x}\frac{\gamma
^{\rho \left( \eta +1\right) -1}}{\left( x^{\rho }-\gamma ^{\rho }\right)
^{1-\alpha }}\varphi \left( \gamma \right) \psi \left( \gamma \right) d\gamma
\\
&\geq &\text{ }^{\rho }\mathcal{I}_{\eta ,k}^{\alpha ,\beta }\varphi \left(
x\right) \frac{\rho ^{1-\beta }x^{k}}{\Gamma \left( \alpha \right) }%
\int_{0}^{x}\frac{\gamma ^{\rho \left( \eta +1\right) -1}}{\left( x^{\rho
}-\gamma ^{\rho }\right) ^{1-\alpha }}\psi \left( \gamma \right) d\gamma \\
&&+\text{ }^{\rho }\mathcal{I}_{\eta ,k}^{\alpha ,\beta }\psi \left(
x\right) \frac{\rho ^{1-\beta }x^{k}}{\Gamma \left( \alpha \right) }%
\int_{0}^{x}\frac{\gamma ^{\rho \left( \eta +1\right) -1}}{\left( x^{\rho
}-\gamma ^{\rho }\right) ^{1-\alpha }}\varphi \left( \gamma \right) d\gamma .
\end{eqnarray*}%
So we have%
\begin{eqnarray*}
&&^{\rho }\mathcal{I}_{\eta ,k}^{\alpha ,\beta }\left( \varphi \psi \right)
\left( x\right) ^{\rho }\mathcal{I}_{\eta ,k}^{\alpha ,\beta }\left(
1\right) +\Lambda _{x,k}^{\rho ,\beta }\left( \alpha ,\eta \right) ^{\rho }%
\mathcal{I}_{\eta ,k}^{\alpha ,\beta }\left( \varphi \psi \right) \left(
x\right) \\
&\geq &\text{ }^{\rho }\mathcal{I}_{\eta ,k}^{\alpha ,\beta }\varphi \left(
x\right) \text{ }^{\rho }\mathcal{I}_{\eta ,k}^{\alpha ,\beta }\psi \left(
x\right) +\text{ }^{\rho }\mathcal{I}_{\eta ,k}^{\alpha ,\beta }\psi \left(
x\right) \text{ }^{\rho }\mathcal{I}_{\eta ,k}^{\alpha ,\beta }\varphi
\left( x\right) .
\end{eqnarray*}%
Hence%
\begin{equation*}
^{\rho }\mathcal{I}_{\eta ,k}^{\alpha ,\beta }\left( \varphi \psi \right)
\left( x\right) \geq \frac{1}{\Lambda _{x,k}^{\rho ,\beta }\left( \alpha
,\eta \right) }\text{ }^{\rho }\mathcal{I}_{\eta ,k}^{\alpha ,\beta }\varphi
\left( x\right) \text{ }^{\rho }\mathcal{I}_{\eta ,k}^{\alpha ,\beta }\psi
\left( x\right) .
\end{equation*}%
The result has proved.
\end{proof}

Our next theorem on Chebyshev functional deals with different fractional
parameters:

\begin{theorem}
Let $\varphi $ and $\psi $ be two integrable and synchronous functions on $%
\left[ 0,\infty \right) .$Then for all $x>0,\alpha >0,\delta >0,\rho
>0,k,\beta ,\lambda ,\eta \in 
\mathbb{R}
,$ we have:%
\begin{eqnarray*}
&&\Lambda _{x,k}^{\rho ,\lambda }\left( \delta ,\eta \right) \text{ }^{\rho }%
\mathcal{I}_{\eta ,k}^{\alpha ,\beta }\left( \varphi \psi \right) \left(
x\right) +\Lambda _{x,k}^{\rho ,\beta }\left( \alpha ,\eta \right) ^{\rho }%
\mathcal{I}_{\eta ,k}^{\delta ,\lambda }\left( \varphi \psi \right) \left(
x\right) \\
&\geq &\text{ }^{\rho }\mathcal{I}_{\eta ,k}^{\alpha ,\beta }\varphi \left(
x\right) \text{ }^{\rho }\mathcal{I}_{\eta ,k}^{\delta ,\lambda }\psi \left(
x\right) +\text{ }^{\rho }\mathcal{I}_{\eta ,k}^{\alpha ,\beta }\psi \left(
x\right) \text{ }^{\rho }\mathcal{I}_{\eta ,k}^{\delta ,\lambda }\varphi
\left( x\right) .
\end{eqnarray*}
\end{theorem}

\begin{proof}
Since $\varphi $ and $\psi $ are synchronous functions on $\left[ 0,\infty
\right) ,$ by similar arguments as in the proof of Theorem (\ref{thm1}) we
can write 
\begin{eqnarray}
&&^{\rho }\mathcal{I}_{\eta ,k}^{\alpha ,\beta }\left( \varphi \psi \right)
\left( x\right) +\Lambda _{x,k}^{\rho ,\beta }\left( \alpha ,\eta \right)
\varphi \left( \gamma \right) \psi \left( \gamma \right)  \notag \\
&\geq &\psi \left( \gamma \right) \text{ }^{\rho }\mathcal{I}_{\eta
,k}^{\alpha ,\beta }\varphi \left( x\right) +\varphi \left( \gamma \right) 
\text{ }^{\rho }\mathcal{I}_{\eta ,k}^{\alpha ,\beta }\psi \left( x\right) .
\label{inq6}
\end{eqnarray}%
Now multiplying both sides of (\ref{inq6}) by $\frac{\rho ^{1-\lambda }x^{k}%
}{\Gamma \left( \delta \right) }\frac{\gamma ^{\rho \left( \eta +1\right) -1}%
}{\left( x^{\rho }-\gamma ^{\rho }\right) ^{1-\delta }},$ where $\gamma \in
\left( 0,x\right) ,$ then integrating the resulting inequality over $\left(
0,x\right) ,$ with respect to the variable $\gamma ,$ we obtain:%
\begin{eqnarray}
&&^{\rho }\mathcal{I}_{\eta ,k}^{\alpha ,\beta }\left( \varphi \psi \right)
\left( x\right) \frac{\rho ^{1-\lambda }x^{k}}{\Gamma \left( \delta \right) }%
\int_{0}^{x}\frac{\gamma ^{\rho \left( \eta +1\right) -1}}{\left( x^{\rho
}-\gamma ^{\rho }\right) ^{1-\delta }}d\gamma  \notag \\
&&+\Lambda _{x,k}^{\rho ,\beta }\left( \alpha ,\eta \right) \frac{\rho
^{1-\lambda }x^{k}}{\Gamma \left( \delta \right) }\int_{0}^{x}\frac{\gamma
^{\rho \left( \eta +1\right) -1}}{\left( x^{\rho }-\gamma ^{\rho }\right)
^{1-\delta }}\varphi \left( \gamma \right) \psi \left( \gamma \right) d\gamma
\notag \\
&\geq &\text{ }^{\rho }\mathcal{I}_{\eta ,k}^{\alpha ,\beta }\varphi \left(
x\right) \frac{\rho ^{1-\lambda }x^{k}}{\Gamma \left( \delta \right) }%
\int_{0}^{x}\frac{\gamma ^{\rho \left( \eta +1\right) -1}}{\left( x^{\rho
}-\gamma ^{\rho }\right) ^{1-\delta }}\psi \left( \gamma \right) d\gamma 
\notag \\
&&+\text{ }^{\rho }\mathcal{I}_{\eta ,k}^{\alpha ,\beta }\psi \left(
x\right) \frac{\rho ^{1-\lambda }x^{k}}{\Gamma \left( \delta \right) }%
\int_{0}^{x}\frac{\gamma ^{\rho \left( \eta +1\right) -1}}{\left( x^{\rho
}-\gamma ^{\rho }\right) ^{1-\delta }}\varphi \left( \gamma \right) d\gamma .
\label{inq8}
\end{eqnarray}%
So we have%
\begin{eqnarray*}
&&\Lambda _{x,k}^{\rho ,\lambda }\left( \delta ,\eta \right) \text{ }^{\rho }%
\mathcal{I}_{\eta ,k}^{\alpha ,\beta }\left( \varphi \psi \right) \left(
x\right) +\Lambda _{x,k}^{\rho ,\beta }\left( \alpha ,\eta \right) ^{\rho }%
\mathcal{I}_{\eta ,k}^{\delta ,\lambda }\left( \varphi \psi \right) \left(
x\right) \\
&\geq &\text{ }^{\rho }\mathcal{I}_{\eta ,k}^{\alpha ,\beta }\varphi \left(
x\right) \text{ }^{\rho }\mathcal{I}_{\eta ,k}^{\delta ,\lambda }\psi \left(
x\right) +\text{ }^{\rho }\mathcal{I}_{\eta ,k}^{\alpha ,\beta }\psi \left(
x\right) \text{ }^{\rho }\mathcal{I}_{\eta ,k}^{\delta ,\lambda }\varphi
\left( x\right) .
\end{eqnarray*}%
Hence the proof.
\end{proof}

\section{\protect\small Generalized fractional inequality for extended
Chebyshev's functional:}

\ \ \ In this section, we consider the extended Chebyshev functional in case
of synchronous functions (\ref{fun2}). To prove our theorem in this section,
we prove the following lemma:

\begin{lemma}
\label{lem1}Let $\varphi $ and $\psi $ be two integrable and synchronous
functions on $\left[ 0,\infty \right) ,$ suppose $s,v:\left[ 0,\infty
\right) \rightarrow \left[ 0,\infty \right) ,$ Then for all $x>0,\alpha
>0,\rho >0,k,\beta ,\eta \in 
\mathbb{R}
$ we have:%
\begin{eqnarray*}
&&^{\rho }\mathcal{I}_{\eta ,k}^{\alpha ,\beta }\left( s\varphi \psi \right)
\left( x\right) ^{\rho }\mathcal{I}_{\eta ,k}^{\alpha ,\beta }v\left(
x\right) +\text{ }^{\rho }\mathcal{I}_{\eta ,k}^{\alpha ,\beta }s\left(
x\right) ^{\rho }\mathcal{I}_{\eta ,k}^{\alpha ,\beta }\left( v\varphi \psi
\right) \left( x\right) \\
&\geq &\text{ }^{\rho }\mathcal{I}_{\eta ,k}^{\alpha ,\beta }\left( s\varphi
\right) \left( x\right) \text{ }^{\rho }\mathcal{I}_{\eta ,k}^{\alpha ,\beta
}\left( v\psi \right) \left( x\right) +\text{ }^{\rho }\mathcal{I}_{\eta
,k}^{\alpha ,\beta }\left( s\psi \right) \left( x\right) \text{ }^{\rho }%
\mathcal{I}_{\eta ,k}^{\alpha ,\beta }\left( v\varphi \right) \left(
x\right) .
\end{eqnarray*}
\end{lemma}

\begin{proof}
For a synchronous functions $\varphi ,\psi $ on $\left[ 0,\infty \right) ,$
we have for all $\tau \geq 0,\gamma \geq 0:$%
\begin{equation*}
\left( \varphi \left( \tau \right) -\varphi \left( \gamma \right) \right)
\left( \psi \left( \tau \right) -\psi \left( \gamma \right) \right) \geq 0.
\end{equation*}%
Therefore%
\begin{equation}
\varphi \left( \tau \right) \psi \left( \tau \right) +\varphi \left( \gamma
\right) \psi \left( \gamma \right) \geq \varphi \left( \tau \right) \psi
\left( \gamma \right) +\varphi \left( \gamma \right) \psi \left( \tau
\right) .  \label{INQ2}
\end{equation}%
Now multiplying both sides of (\ref{INQ2}) by $\frac{\rho ^{1-\beta }x^{k}}{%
\Gamma \left( \alpha \right) }\frac{\tau ^{\rho \left( \eta +1\right)
-1}s\left( \tau \right) }{\left( x^{\rho }-\tau ^{\rho }\right) ^{1-\alpha }}%
,$ where $\tau \in \left( 0,x\right) ,$then integrating the resulting
inequality over $\left( 0,x\right) ,$ with respect to the variable $\tau ,$
we get:%
\begin{eqnarray*}
&&\frac{\rho ^{1-\beta }x^{k}}{\Gamma \left( \alpha \right) }\int_{0}^{x}%
\frac{\tau ^{\rho \left( \eta +1\right) -1}}{\left( x^{\rho }-\tau ^{\rho
}\right) ^{1-\alpha }}s\left( \tau \right) \varphi \left( \tau \right) \psi
\left( \tau \right) d\tau \\
&&+\varphi \left( \gamma \right) \psi \left( \gamma \right) \frac{\rho
^{1-\beta }x^{k}}{\Gamma \left( \alpha \right) }\int_{0}^{x}\frac{\tau
^{\rho \left( \eta +1\right) -1}}{\left( x^{\rho }-\tau ^{\rho }\right)
^{1-\alpha }}s\left( \tau \right) d\tau \\
&\geq &\psi \left( \gamma \right) \frac{\rho ^{1-\beta }x^{k}}{\Gamma \left(
\alpha \right) }\int_{0}^{x}\frac{\tau ^{\rho \left( \eta +1\right) -1}}{%
\left( x^{\rho }-\tau ^{\rho }\right) ^{1-\alpha }}s\left( \tau \right)
\varphi \left( \tau \right) d\tau \\
&&+\varphi \left( \gamma \right) \frac{\rho ^{1-\beta }x^{k}}{\Gamma \left(
\alpha \right) }\int_{0}^{x}\frac{\tau ^{\rho \left( \eta +1\right) -1}}{%
\left( x^{\rho }-\tau ^{\rho }\right) ^{1-\alpha }}s\left( \tau \right) \psi
\left( \tau \right) d\tau .
\end{eqnarray*}%
So we have%
\begin{eqnarray}
&&^{\rho }\mathcal{I}_{\eta ,k}^{\alpha ,\beta }\left( s\varphi \psi \right)
\left( x\right) +\varphi \left( \gamma \right) \psi \left( \gamma \right) 
\text{ }^{\rho }\mathcal{I}_{\eta ,k}^{\alpha ,\beta }s\left( x\right) 
\notag \\
&\geq &\psi \left( \gamma \right) \text{ }^{\rho }\mathcal{I}_{\eta
,k}^{\alpha ,\beta }\left( s\varphi \right) \left( x\right) +\varphi \left(
\gamma \right) \text{ }^{\rho }\mathcal{I}_{\eta ,k}^{\alpha ,\beta }\left(
s\psi \right) \left( x\right) .  \label{INQ4}
\end{eqnarray}%
Now multiplying both sides of (\ref{INQ4}) by $\frac{\rho ^{1-\beta }x^{k}}{%
\Gamma \left( \alpha \right) }\frac{\gamma ^{\rho \left( \eta +1\right)
-1}v\left( \gamma \right) }{\left( x^{\rho }-\gamma ^{\rho }\right)
^{1-\alpha }},$ where $\gamma \in \left( 0,x\right) ,$ and integrating the
resulting inequality over $\left( 0,x\right) ,$ with respect to the variable 
$\gamma ,$ we obtain:%
\begin{eqnarray*}
&&^{\rho }\mathcal{I}_{\eta ,k}^{\alpha ,\beta }\left( s\varphi \psi \right)
\left( x\right) \frac{\rho ^{1-\beta }x^{k}}{\Gamma \left( \alpha \right) }%
\int_{0}^{x}\frac{\gamma ^{\rho \left( \eta +1\right) -1}}{\left( x^{\rho
}-\gamma ^{\rho }\right) ^{1-\alpha }}v\left( \gamma \right) d\gamma \\
&&+\text{ }^{\rho }\mathcal{I}_{\eta ,k}^{\alpha ,\beta }s\left( x\right) 
\frac{\rho ^{1-\beta }x^{k}}{\Gamma \left( \alpha \right) }\int_{0}^{x}\frac{%
\gamma ^{\rho \left( \eta +1\right) -1}}{\left( x^{\rho }-\gamma ^{\rho
}\right) ^{1-\alpha }}v\left( \gamma \right) \varphi \left( \gamma \right)
\psi \left( \gamma \right) d\gamma \\
&\geq &\text{ }^{\rho }\mathcal{I}_{\eta ,k}^{\alpha ,\beta }\left( s\varphi
\right) \left( x\right) \frac{\rho ^{1-\beta }x^{k}}{\Gamma \left( \alpha
\right) }\int_{0}^{x}\frac{\gamma ^{\rho \left( \eta +1\right) -1}}{\left(
x^{\rho }-\gamma ^{\rho }\right) ^{1-\alpha }}v\left( \gamma \right) \psi
\left( \gamma \right) d\gamma \\
&&+\text{ }^{\rho }\mathcal{I}_{\eta ,k}^{\alpha ,\beta }\left( s\psi
\right) \left( x\right) \frac{\rho ^{1-\beta }x^{k}}{\Gamma \left( \alpha
\right) }\int_{0}^{x}\frac{\gamma ^{\rho \left( \eta +1\right) -1}}{\left(
x^{\rho }-\gamma ^{\rho }\right) ^{1-\alpha }}v\left( \gamma \right) \varphi
\left( \gamma \right) d\gamma .
\end{eqnarray*}%
Therefore%
\begin{eqnarray*}
&&^{\rho }\mathcal{I}_{\eta ,k}^{\alpha ,\beta }\left( s\varphi \psi \right)
\left( x\right) \text{ }^{\rho }\mathcal{I}_{\eta ,k}^{\alpha ,\beta
}v\left( x\right) +\text{ }^{\rho }\mathcal{I}_{\eta ,k}^{\alpha ,\beta
}s\left( x\right) \text{ }^{\rho }\mathcal{I}_{\eta ,k}^{\alpha ,\beta
}\left( v\varphi \psi \right) \left( x\right) \\
&\geq &\text{ }^{\rho }\mathcal{I}_{\eta ,k}^{\alpha ,\beta }\left( s\varphi
\right) \left( x\right) \text{ }^{\rho }\mathcal{I}_{\eta ,k}^{\alpha ,\beta
}\left( v\psi \right) \left( x\right) +\text{ }^{\rho }\mathcal{I}_{\eta
,k}^{\alpha ,\beta }\left( s\psi \right) \left( x\right) \text{ }^{\rho }%
\mathcal{I}_{\eta ,k}^{\alpha ,\beta }\left( v\varphi \right) \left(
x\right) .
\end{eqnarray*}%
Hence the proof.
\end{proof}

\begin{theorem}
\label{Thm3} Let $\varphi $ and $\psi $ be two integrable and synchronous
functions on $\left[ 0,\infty \right) ,$ and suppose $f,g,h:\left[ 0,\infty
\right) \rightarrow \left[ 0,\infty \right) .$ Then for all $x>0,\alpha
>0,\rho >0,k,\beta ,\eta \in 
\mathbb{R}
$ we have:%
\begin{eqnarray}
&&^{\rho }\mathcal{I}_{\eta ,k}^{\alpha ,\beta }h\left( x\right) [^{\rho }%
\mathcal{I}_{\eta ,k}^{\alpha ,\beta }\left( f\varphi \psi \right) \left(
x\right) \text{ }^{\rho }\mathcal{I}_{\eta ,k}^{\alpha ,\beta }g\left(
x\right) +2\text{ }^{\rho }\mathcal{I}_{\eta ,k}^{\alpha ,\beta }f\left(
x\right) \text{ }^{\rho }\mathcal{I}_{\eta ,k}^{\alpha ,\beta }\left(
g\varphi \psi \right) \left( x\right)  \notag \\
&&+\text{ }^{\rho }\mathcal{I}_{\eta ,k}^{\alpha ,\beta }g\left( x\right) 
\text{ }^{\rho }\mathcal{I}_{\eta ,k}^{\alpha ,\beta }\left( f\varphi \psi
\right) \left( x\right) ]+2\text{ }^{\rho }\mathcal{I}_{\eta ,k}^{\alpha
,\beta }f\left( x\right) \text{ }^{\rho }\mathcal{I}_{\eta ,k}^{\alpha
,\beta }\left( h\varphi \psi \right) \left( x\right) \text{ }^{\rho }%
\mathcal{I}_{\eta ,k}^{\alpha ,\beta }g\left( x\right)  \notag \\
&\geq &\text{ }^{\rho }\mathcal{I}_{\eta ,k}^{\alpha ,\beta }h\left(
x\right) \left[ \text{ }^{\rho }\mathcal{I}_{\eta ,k}^{\alpha ,\beta }\left(
f\varphi \right) \left( x\right) \text{ }^{\rho }\mathcal{I}_{\eta
,k}^{\alpha ,\beta }\left( g\psi \right) \left( x\right) +\text{ }^{\rho }%
\mathcal{I}_{\eta ,k}^{\alpha ,\beta }\left( f\psi \right) \left( x\right) 
\text{ }^{\rho }\mathcal{I}_{\eta ,k}^{\alpha ,\beta }\left( g\varphi
\right) \left( x\right) \right]  \notag \\
&&+\text{ }^{\rho }\mathcal{I}_{\eta ,k}^{\alpha ,\beta }f\left( x\right) %
\left[ \text{ }^{\rho }\mathcal{I}_{\eta ,k}^{\alpha ,\beta }\left( h\varphi
\right) \left( x\right) \text{ }^{\rho }\mathcal{I}_{\eta ,k}^{\alpha ,\beta
}\left( g\psi \right) \left( x\right) +\text{ }^{\rho }\mathcal{I}_{\eta
,k}^{\alpha ,\beta }\left( h\psi \right) \left( x\right) \text{ }^{\rho }%
\mathcal{I}_{\eta ,k}^{\alpha ,\beta }\left( g\varphi \right) \left(
x\right) \right]  \notag \\
&&+\text{ }^{\rho }\mathcal{I}_{\eta ,k}^{\alpha ,\beta }g\left( x\right) %
\left[ \text{ }^{\rho }\mathcal{I}_{\eta ,k}^{\alpha ,\beta }\left( h\varphi
\right) \left( x\right) \text{ }^{\rho }\mathcal{I}_{\eta ,k}^{\alpha ,\beta
}\left( f\psi \right) \left( x\right) +\text{ }^{\rho }\mathcal{I}_{\eta
,k}^{\alpha ,\beta }\left( h\psi \right) \left( x\right) \text{ }^{\rho }%
\mathcal{I}_{\eta ,k}^{\alpha ,\beta }\left( f\varphi \right) \left(
x\right) \right] .  \notag \\
&&  \label{INQ14}
\end{eqnarray}
\end{theorem}

\begin{proof}
In lemma (\ref{lem1}) putting $s=f,v=g,$ then multiplying both sides of the
resulting inequality by $^{\rho }\mathcal{I}_{\eta ,k}^{\alpha ,\beta
}h\left( x\right) ,$ we get:%
\begin{eqnarray}
&&^{\rho }\mathcal{I}_{\eta ,k}^{\alpha ,\beta }h\left( x\right) \text{ }%
^{\rho }\mathcal{I}_{\eta ,k}^{\alpha ,\beta }\left( f\varphi \psi \right)
\left( x\right) \text{ }^{\rho }\mathcal{I}_{\eta ,k}^{\alpha ,\beta
}g\left( x\right)  \notag \\
&&+\text{ }^{\rho }\mathcal{I}_{\eta ,k}^{\alpha ,\beta }h\left( x\right) 
\text{ }^{\rho }\mathcal{I}_{\eta ,k}^{\alpha ,\beta }f\left( x\right) \text{
}^{\rho }\mathcal{I}_{\eta ,k}^{\alpha ,\beta }\left( g\varphi \psi \right)
\left( x\right)  \notag \\
&\geq &\text{ }^{\rho }\mathcal{I}_{\eta ,k}^{\alpha ,\beta }h\left(
x\right) \text{ }^{\rho }\mathcal{I}_{\eta ,k}^{\alpha ,\beta }\left(
f\varphi \right) \left( x\right) \text{ }^{\rho }\mathcal{I}_{\eta
,k}^{\alpha ,\beta }\left( g\psi \right) \left( x\right)  \label{INQ8} \\
&&+\text{ }^{\rho }\mathcal{I}_{\eta ,k}^{\alpha ,\beta }h\left( x\right) 
\text{ }^{\rho }\mathcal{I}_{\eta ,k}^{\alpha ,\beta }\left( f\psi \right)
\left( x\right) \text{ }^{\rho }\mathcal{I}_{\eta ,k}^{\alpha ,\beta }\left(
g\varphi \right) \left( x\right) .  \notag
\end{eqnarray}%
Now putting $s=h,v=g,$ in lemma (\ref{lem1}) and multiplying both sides of
the resulting inequality by $^{\rho }\mathcal{I}_{\eta ,k}^{\alpha ,\beta
}f\left( x\right) ,$ we obtain: 
\begin{eqnarray}
&&^{\rho }\mathcal{I}_{\eta ,k}^{\alpha ,\beta }f\left( x\right) \text{ }%
^{\rho }\mathcal{I}_{\eta ,k}^{\alpha ,\beta }\left( h\varphi \psi \right)
\left( x\right) \text{ }^{\rho }\mathcal{I}_{\eta ,k}^{\alpha ,\beta
}g\left( x\right)  \notag \\
&&+\text{ }^{\rho }\mathcal{I}_{\eta ,k}^{\alpha ,\beta }f\left( x\right) 
\text{ }^{\rho }\mathcal{I}_{\eta ,k}^{\alpha ,\beta }h\left( x\right) \text{
}^{\rho }\mathcal{I}_{\eta ,k}^{\alpha ,\beta }\left( g\varphi \psi \right)
\left( x\right)  \notag \\
&\geq &\text{ }^{\rho }\mathcal{I}_{\eta ,k}^{\alpha ,\beta }f\left(
x\right) \text{ }^{\rho }\mathcal{I}_{\eta ,k}^{\alpha ,\beta }\left(
h\varphi \right) \left( x\right) \text{ }^{\rho }\mathcal{I}_{\eta
,k}^{\alpha ,\beta }\left( g\psi \right) \left( x\right)  \label{INQ10} \\
&&+\text{ }^{\rho }\mathcal{I}_{\eta ,k}^{\alpha ,\beta }f\left( x\right) 
\text{ }^{\rho }\mathcal{I}_{\eta ,k}^{\alpha ,\beta }\left( h\psi \right)
\left( x\right) \text{ }^{\rho }\mathcal{I}_{\eta ,k}^{\alpha ,\beta }\left(
g\varphi \right) \left( x\right) .  \notag
\end{eqnarray}%
Now putting $s=h,v=f,$ in lemma (\ref{lem1}),then multiplying both sides of
the resulting inequality by $^{\rho }\mathcal{I}_{\eta ,k}^{\alpha ,\beta
}g\left( x\right) ,$ we get:%
\begin{eqnarray}
&^{\rho }\mathcal{I}_{\eta ,k}^{\alpha ,\beta }g\left( x\right) &^{\rho }%
\mathcal{I}_{\eta ,k}^{\alpha ,\beta }\left( h\varphi \psi \right) \left(
x\right) \text{ }^{\rho }\mathcal{I}_{\eta ,k}^{\alpha ,\beta }f\left(
x\right)  \notag \\
&&+\text{ }^{\rho }\mathcal{I}_{\eta ,k}^{\alpha ,\beta }g\left( x\right) 
\text{ }^{\rho }\mathcal{I}_{\eta ,k}^{\alpha ,\beta }h\left( x\right) \text{
}^{\rho }\mathcal{I}_{\eta ,k}^{\alpha ,\beta }\left( f\varphi \psi \right)
\left( x\right)  \notag \\
&\geq &\text{ }^{\rho }\mathcal{I}_{\eta ,k}^{\alpha ,\beta }g\left(
x\right) \text{ }^{\rho }\mathcal{I}_{\eta ,k}^{\alpha ,\beta }\left(
h\varphi \right) \left( x\right) \text{ }^{\rho }\mathcal{I}_{\eta
,k}^{\alpha ,\beta }\left( f\psi \right) \left( x\right)  \label{INQ12} \\
&&+\text{ }^{\rho }\mathcal{I}_{\eta ,k}^{\alpha ,\beta }g\left( x\right) 
\text{ }^{\rho }\mathcal{I}_{\eta ,k}^{\alpha ,\beta }\left( h\psi \right)
\left( x\right) \text{ }^{\rho }\mathcal{I}_{\eta ,k}^{\alpha ,\beta }\left(
f\varphi \right) \left( x\right) .  \notag
\end{eqnarray}%
By adding the inequalities(\ref{INQ8}, \ref{INQ10}, \ref{INQ12}) we get the
inequality (\ref{INQ14}).
\end{proof}

Now we give the lemma required for proving our next theorem for different
parameter.

\begin{lemma}
\label{lem2} Let $\varphi $ and $\psi $ be two integrable and synchronous
functions on $\left[ 0,\infty \right) ,$ suppose that $s,v:\left[ 0,\infty
\right) \rightarrow \left[ 0,\infty \right) ,$ then for all $x>0,\alpha
>0,\delta >0,\rho >0,k,\beta ,\lambda ,\eta \in 
\mathbb{R}
,$ we have:%
\begin{eqnarray*}
&&^{\rho }\mathcal{I}_{\eta ,k}^{\alpha ,\beta }\left( s\varphi \psi \right)
\left( x\right) ^{\rho }\mathcal{I}_{\eta ,k}^{\delta ,\lambda }v\left(
x\right) +\text{ }^{\rho }\mathcal{I}_{\eta ,k}^{\alpha ,\beta }s\left(
x\right) \text{ }^{\rho }\mathcal{I}_{\eta ,k}^{\delta ,\lambda }\left(
v\varphi \psi \right) \left( x\right)  \\
&\geq &\text{ }^{\rho }\mathcal{I}_{\eta ,k}^{\alpha ,\beta }\left( s\varphi
\right) \left( x\right) \text{ }^{\rho }\mathcal{I}_{\eta ,k}^{\delta
,\lambda }\left( v\psi \right) \left( x\right) +\text{ }^{\rho }\mathcal{I}%
_{\eta ,k}^{\alpha ,\beta }\left( s\psi \right) \left( x\right) \text{ }%
^{\rho }\mathcal{I}_{\eta ,k}^{\delta ,\lambda }\left( v\varphi \right)
\left( x\right) .
\end{eqnarray*}
\end{lemma}

\begin{proof}
Since $\varphi $ and $\psi $ are synchronous functions on $\left[ 0,\infty
\right) ,$ by similar arguments as in the proof of lemma (\ref{lem1}), we
can write%
\begin{eqnarray}
&&^{\rho }\mathcal{I}_{\eta ,k}^{\alpha ,\beta }\left( s\varphi \psi \right)
\left( x\right) +\varphi \left( \gamma \right) \psi \left( \gamma \right) 
\text{ }^{\rho }\mathcal{I}_{\eta ,k}^{\alpha ,\beta }s\left( x\right) 
\notag \\
&\geq &\psi \left( \gamma \right) \text{ }^{\rho }\mathcal{I}_{\eta
,k}^{\alpha ,\beta }\left( s\varphi \right) \left( x\right) +\varphi \left(
\gamma \right) \text{ }^{\rho }\mathcal{I}_{\eta ,k}^{\alpha ,\beta }\left(
s\psi \right) \left( x\right) .  \label{INQ15}
\end{eqnarray}%
Multiplying both sides of (\ref{INQ15}) by $\frac{\rho ^{1-\lambda }x^{k}}{%
\Gamma \left( \delta \right) }\frac{\gamma ^{\rho \left( \eta +1\right)
-1}v\left( \gamma \right) }{\left( x^{\rho }-\gamma ^{\rho }\right)
^{1-\delta }},$ where $\gamma \in \left( 0,x\right) ,$then integrating
resulting inequality over $\left( 0,x\right) ,$ with respect to the variable 
$\gamma ,$ we obtain:%
\begin{eqnarray*}
&&^{\rho }\mathcal{I}_{\eta ,k}^{\alpha ,\beta }\left( s\varphi \psi \right)
\left( x\right) \frac{\rho ^{1-\lambda }x^{k}}{\Gamma \left( \delta \right) }%
\int_{0}^{x}\frac{\gamma ^{\rho \left( \eta +1\right) -1}}{\left( x^{\rho
}-\gamma ^{\rho }\right) ^{1-\delta }}v\left( \gamma \right) d\gamma \\
&&+\text{ }^{\rho }\mathcal{I}_{\eta ,k}^{\alpha ,\beta }s\left( x\right) 
\frac{\rho ^{1-\lambda }x^{k}}{\Gamma \left( \delta \right) }\int_{0}^{x}%
\frac{\gamma ^{\rho \left( \eta +1\right) -1}}{\left( x^{\rho }-\gamma
^{\rho }\right) ^{1-\delta }}v\left( \gamma \right) \varphi \left( \gamma
\right) \psi \left( \gamma \right) d\gamma \\
&\geq &\text{ }^{\rho }\mathcal{I}_{\eta ,k}^{\alpha ,\beta }\left( s\varphi
\right) \left( x\right) \frac{\rho ^{1-\lambda }x^{k}}{\Gamma \left( \delta
\right) }\int_{0}^{x}\frac{\gamma ^{\rho \left( \eta +1\right) -1}}{\left(
x^{\rho }-\gamma ^{\rho }\right) ^{1-\delta }}v\left( \gamma \right) \psi
\left( \gamma \right) d\gamma \\
&&+\text{ }^{\rho }\mathcal{I}_{\eta ,k}^{\alpha ,\beta }\left( s\psi
\right) \left( x\right) \frac{\rho ^{1-\lambda }x^{k}}{\Gamma \left( \delta
\right) }\int_{0}^{x}\frac{\gamma ^{\rho \left( \eta +1\right) -1}}{\left(
x^{\rho }-\gamma ^{\rho }\right) ^{1-\delta }}v\left( \gamma \right) \varphi
\left( \gamma \right) d\gamma .
\end{eqnarray*}%
Therefore%
\begin{eqnarray*}
&&^{\rho }\mathcal{I}_{\eta ,k}^{\alpha ,\beta }\left( s\varphi \psi \right)
\left( x\right) \text{ }^{\rho }\mathcal{I}_{\eta ,k}^{\delta ,\lambda
}v\left( x\right) +\text{ }^{\rho }\mathcal{I}_{\eta ,k}^{\alpha ,\beta
}s\left( x\right) \text{ }^{\rho }\mathcal{I}_{\eta ,k}^{\delta ,\lambda
}\left( v\varphi \psi \right) \left( x\right) \\
&\geq &\text{ }^{\rho }\mathcal{I}_{\eta ,k}^{\alpha ,\beta }\left( s\varphi
\right) \left( x\right) \text{ }^{\rho }\mathcal{I}_{\eta ,k}^{\delta
,\lambda }\left( v\psi \right) \left( x\right) +\text{ }^{\rho }\mathcal{I}%
_{\eta ,k}^{\alpha ,\beta }\left( s\psi \right) \left( x\right) \text{ }%
^{\rho }\mathcal{I}_{\eta ,k}^{\delta ,\lambda }\left( v\varphi \right)
\left( x\right) .
\end{eqnarray*}%
Hence the proof.
\end{proof}

\begin{theorem}
Let $\varphi $ and $\psi $ be two integrable and synchronous functions on $%
\left[ 0,\infty \right) ,$ and suppose $f,g,h:\left[ 0,\infty \right)
\rightarrow \left[ 0,\infty \right) .$ Then for all $x>0,\alpha >0,\delta
>0,\rho >0,k,\beta ,\lambda ,\eta \in 
\mathbb{R}
,$ we have:%
\begin{eqnarray}
&&^{\rho }\mathcal{I}_{\eta ,k}^{\alpha ,\beta }h\left( x\right) [\text{ }%
^{\rho }\mathcal{I}_{\eta ,k}^{\alpha ,\beta }\left( f\varphi \psi \right)
\left( x\right) \text{ }^{\rho }\mathcal{I}_{\eta ,k}^{\delta ,\lambda
}g\left( x\right) +2\text{ }^{\rho }\mathcal{I}_{\eta ,k}^{\alpha ,\beta
}f\left( x\right) \text{ }^{\rho }\mathcal{I}_{\eta ,k}^{\delta ,\lambda
}\left( g\varphi \psi \right) \left( x\right)  \notag \\
&&+\text{ }^{\rho }\mathcal{I}_{\eta ,k}^{\alpha ,\beta }g\left( x\right) 
\text{ }^{\rho }\mathcal{I}_{\eta ,k}^{\delta ,\lambda }\left( f\varphi \psi
\right) \left( x\right) ]  \notag \\
&&+\left[ \text{ }^{\rho }\mathcal{I}_{\eta ,k}^{\alpha ,\beta }f\left(
x\right) \text{ }^{\rho }\mathcal{I}_{\eta ,k}^{\delta ,\lambda }g\left(
x\right) +\text{ }^{\rho }\mathcal{I}_{\eta ,k}^{\alpha ,\beta }g\left(
x\right) \text{ }^{\rho }\mathcal{I}_{\eta ,k}^{\delta ,\lambda }f\left(
x\right) \right] \text{ }^{\rho }\mathcal{I}_{\eta ,k}^{\alpha ,\beta
}\left( h\varphi \psi \right) \left( x\right)  \notag \\
&\geq &\text{ }^{\rho }\mathcal{I}_{\eta ,k}^{\alpha ,\beta }h\left(
x\right) \left[ \text{ }^{\rho }\mathcal{I}_{\eta ,k}^{\alpha ,\beta }\left(
f\varphi \right) \left( x\right) \text{ }^{\rho }\mathcal{I}_{\eta
,k}^{\delta ,\lambda }\left( g\psi \right) \left( x\right) +\text{ }^{\rho }%
\mathcal{I}_{\eta ,k}^{\alpha ,\beta }\left( f\psi \right) \left( x\right) 
\text{ }^{\rho }\mathcal{I}_{\eta ,k}^{\delta ,\lambda }\left( g\varphi
\right) \left( x\right) \right]  \notag \\
&&+\text{ }^{\rho }\mathcal{I}_{\eta ,k}^{\alpha ,\beta }f\left( x\right) %
\left[ \text{ }^{\rho }\mathcal{I}_{\eta ,k}^{\alpha ,\beta }\left( h\varphi
\right) \left( x\right) \text{ }^{\rho }\mathcal{I}_{\eta ,k}^{\delta
,\lambda }\left( g\psi \right) \left( x\right) +\text{ }^{\rho }\mathcal{I}%
_{\eta ,k}^{\alpha ,\beta }\left( h\psi \right) \left( x\right) \text{ }%
^{\rho }\mathcal{I}_{\eta ,k}^{\delta ,\lambda }\left( g\varphi \right)
\left( x\right) \right]  \notag \\
&&+\text{ }^{\rho }\mathcal{I}_{\eta ,k}^{\alpha ,\beta }g\left( x\right) %
\left[ \text{ }^{\rho }\mathcal{I}_{\eta ,k}^{\alpha ,\beta }\left( h\varphi
\right) \left( x\right) \text{ }^{\rho }\mathcal{I}_{\eta ,k}^{\delta
,\lambda }\left( f\psi \right) \left( x\right) +\text{ }^{\rho }\mathcal{I}%
_{\eta ,k}^{\alpha ,\beta }\left( h\psi \right) \left( x\right) \text{ }%
^{\rho }\mathcal{I}_{\eta ,k}^{\delta ,\lambda }\left( f\varphi \right)
\left( x\right) \right] .  \notag \\
&&  \label{INQ24}
\end{eqnarray}
\end{theorem}

\begin{proof}
In lemma (\ref{lem2}), putting $s=f,v=g,$ we can write:%
\begin{eqnarray}
&&^{\rho }\mathcal{I}_{\eta ,k}^{\alpha ,\beta }\left( f\varphi \psi \right)
\left( x\right) \text{ }^{\rho }\mathcal{I}_{\eta ,k}^{\delta ,\lambda
}g\left( x\right) +\text{ }^{\rho }\mathcal{I}_{\eta ,k}^{\alpha ,\beta
}f\left( x\right) \text{ }^{\rho }\mathcal{I}_{\eta ,k}^{\delta ,\lambda
}\left( g\varphi \psi \right) \left( x\right)   \notag \\
&\geq &\text{ }^{\rho }\mathcal{I}_{\eta ,k}^{\alpha ,\beta }\left( f\varphi
\right) \left( x\right) \text{ }^{\rho }\mathcal{I}_{\eta ,k}^{\delta
,\lambda }\left( g\psi \right) \left( x\right) +\text{ }^{\rho }\mathcal{I}%
_{\eta ,k}^{\alpha ,\beta }\left( f\psi \right) \left( x\right) \text{ }%
^{\rho }\mathcal{I}_{\eta ,k}^{\delta ,\lambda }\left( g\varphi \right)
\left( x\right) .  \label{INQ18}
\end{eqnarray}%
Multiplying both sides of (\ref{INQ18}) by $^{\rho }\mathcal{I}_{\eta
,k}^{\alpha ,\beta }h\left( x\right) ,$ we get%
\begin{eqnarray}
&&^{\rho }\mathcal{I}_{\eta ,k}^{\alpha ,\beta }h\left( x\right) \text{ }%
^{\rho }\mathcal{I}_{\eta ,k}^{\alpha ,\beta }\left( f\varphi \psi \right)
\left( x\right) \text{ }^{\rho }\mathcal{I}_{\eta ,k}^{\delta ,\lambda
}g\left( x\right)   \notag \\
&&+\text{ }^{\rho }\mathcal{I}_{\eta ,k}^{\alpha ,\beta }h\left( x\right) 
\text{ }^{\rho }\mathcal{I}_{\eta ,k}^{\alpha ,\beta }f\left( x\right) \text{
}^{\rho }\mathcal{I}_{\eta ,k}^{\delta ,\lambda }\left( g\varphi \psi
\right) \left( x\right)   \notag \\
&\geq &\text{ }^{\rho }\mathcal{I}_{\eta ,k}^{\alpha ,\beta }h\left(
x\right) \text{ }^{\rho }\mathcal{I}_{\eta ,k}^{\alpha ,\beta }\left(
f\varphi \right) \left( x\right) \text{ }^{\rho }\mathcal{I}_{\eta
,k}^{\delta ,\lambda }\left( g\psi \right) \left( x\right)   \notag \\
&&+\text{ }^{\rho }\mathcal{I}_{\eta ,k}^{\alpha ,\beta }h\left( x\right) 
\text{ }^{\rho }\mathcal{I}_{\eta ,k}^{\alpha ,\beta }\left( f\psi \right)
\left( x\right) \text{ }^{\rho }\mathcal{I}_{\eta ,k}^{\delta ,\lambda
}\left( g\varphi \right) \left( x\right) .  \label{INQ19}
\end{eqnarray}%
Now putting $s=h,v=g,$ in lemma (\ref{lem2}), then multiplying both sides of
the resulting inequality by $^{\rho }\mathcal{I}_{\eta ,k}^{\alpha ,\beta
}f\left( x\right) ,$ we obtain:%
\begin{eqnarray}
&&^{\rho }\mathcal{I}_{\eta ,k}^{\alpha ,\beta }f\left( x\right) \text{ }%
^{\rho }\mathcal{I}_{\eta ,k}^{\alpha ,\beta }\left( h\varphi \psi \right)
\left( x\right) \text{ }^{\rho }\mathcal{I}_{\eta ,k}^{\delta ,\lambda
}g\left( x\right)   \notag \\
&&+\text{ }^{\rho }\mathcal{I}_{\eta ,k}^{\alpha ,\beta }f\left( x\right) 
\text{ }^{\rho }\mathcal{I}_{\eta ,k}^{\alpha ,\beta }h\left( x\right) \text{
}^{\rho }\mathcal{I}_{\eta ,k}^{\delta ,\lambda }\left( g\varphi \psi
\right) \left( x\right)   \notag \\
&\geq &\text{ }^{\rho }\mathcal{I}_{\eta ,k}^{\alpha ,\beta }f\left(
x\right) \text{ }^{\rho }\mathcal{I}_{\eta ,k}^{\alpha ,\beta }\left(
h\varphi \right) \left( x\right) \text{ }^{\rho }\mathcal{I}_{\eta
,k}^{\delta ,\lambda }\left( g\psi \right) \left( x\right)   \notag \\
&&+\text{ }^{\rho }\mathcal{I}_{\eta ,k}^{\alpha ,\beta }f\left( x\right) 
\text{ }^{\rho }\mathcal{I}_{\eta ,k}^{\alpha ,\beta }\left( h\psi \right)
\left( x\right) \text{ }^{\rho }\mathcal{I}_{\eta ,k}^{\delta ,\lambda
}\left( g\varphi \right) \left( x\right) .  \label{INQ21}
\end{eqnarray}%
Now putting $s=h,v=f,$ in lemma (\ref{lem2}) and multiplying both sides of
the resulting inequality by $^{\rho }\mathcal{I}_{\eta ,k}^{\alpha ,\beta
}g\left( x\right) ,$ we get:%
\begin{eqnarray}
&&^{\rho }\mathcal{I}_{\eta ,k}^{\alpha ,\beta }g\left( x\right) \text{ }%
^{\rho }\mathcal{I}_{\eta ,k}^{\alpha ,\beta }\left( h\varphi \psi \right)
\left( x\right) \text{ }^{\rho }\mathcal{I}_{\eta ,k}^{\delta ,\lambda
}f\left( x\right)   \notag \\
&&+\text{ }^{\rho }\mathcal{I}_{\eta ,k}^{\alpha ,\beta }g\left( x\right) 
\text{ }^{\rho }\mathcal{I}_{\eta ,k}^{\alpha ,\beta }h\left( x\right) \text{
}^{\rho }\mathcal{I}_{\eta ,k}^{\delta ,\lambda }\left( f\varphi \psi
\right) \left( x\right)   \notag \\
&\geq &\text{ }^{\rho }\mathcal{I}_{\eta ,k}^{\alpha ,\beta }g\left(
x\right) \text{ }^{\rho }\mathcal{I}_{\eta ,k}^{\alpha ,\beta }\left(
h\varphi \right) \left( x\right) \text{ }^{\rho }\mathcal{I}_{\eta
,k}^{\delta ,\lambda }\left( f\psi \right) \left( x\right)   \notag \\
&&+\text{ }^{\rho }\mathcal{I}_{\eta ,k}^{\alpha ,\beta }g\left( x\right) 
\text{ }^{\rho }\mathcal{I}_{\eta ,k}^{\alpha ,\beta }\left( h\psi \right)
\left( x\right) \text{ }^{\rho }\mathcal{I}_{\eta ,k}^{\delta ,\lambda
}\left( f\varphi \right) \left( x\right) .  \label{INQ23}
\end{eqnarray}%
By adding the inequalities(\ref{INQ19}), (\ref{INQ21}), (\ref{INQ23}) we get
the required inequality (\ref{INQ24}).
\end{proof}

\section{{\protect\small Generalized fractional inequality for Chebyshev
functional for differentiable functions}$:$}

\ \ \ Our results in present section are for Chebyshev functional in case of
differentiable functions whose derivatives belong to $L_{p}\left( \left[
0,\infty \right] \right) $ \cite{FF}.

First we give the following lemma:

\begin{lemma}
\label{lem3} Let $h$ be a positive function on $\left[ 0,\infty \right) ,$
and let $\varphi ,\psi ,$ be two differentiable functions on $\left[
0,\infty \right) .$ Then for all $x>0,\alpha >0,\rho >0,k,\beta ,\eta \in 
\mathbb{R}
$ we have:%
\begin{eqnarray}
&&\frac{\rho ^{2\left( 1-\beta \right) }x^{2k}}{\Gamma ^{2}\left( \alpha
\right) }\int_{0}^{x}\int_{0}^{x}\frac{\tau ^{\rho \left( \eta +1\right) -1}%
}{\left( x^{\rho }-\tau ^{\rho }\right) ^{1-\alpha }}\frac{\gamma ^{\rho
\left( \eta +1\right) -1}}{\left( x^{\rho }-\gamma ^{\rho }\right)
^{1-\alpha }}h\left( \tau \right) h\left( \gamma \right) H\left( \tau
,\gamma \right) d\tau d\gamma  \notag \\
&=&\left[ \text{ }^{\rho }\mathcal{I}_{\eta ,k}^{\alpha ,\beta }\left(
h\varphi \psi \right) \left( x\right) \text{ }^{\rho }\mathcal{I}_{\eta
,k}^{\alpha ,\beta }h\left( x\right) -\text{ }^{\rho }\mathcal{I}_{\eta
,k}^{\alpha ,\beta }\left( h\psi \right) \left( x\right) \text{ }^{\rho }%
\mathcal{I}_{\eta ,k}^{\alpha ,\beta }\left( h\varphi \right) \left(
x\right) \right] .  \label{inq24}
\end{eqnarray}
\end{lemma}

\begin{proof}
Define 
\begin{equation}
H\left( \tau ,\gamma \right) =\left( \varphi \left( \tau \right) -\varphi
\left( \gamma \right) \right) \left( \psi \left( \tau \right) -\psi \left(
\gamma \right) \right) ;\tau ,\gamma \in \left( 0,x\right) ,x>0.
\label{INQ25}
\end{equation}%
Multiplying both sides of (\ref{INQ25}) by $\frac{\rho ^{1-\beta }x^{k}}{%
\Gamma \left( \alpha \right) }\frac{\tau ^{\rho \left( \eta +1\right)
-1}h\left( \tau \right) }{\left( x^{\rho }-\tau ^{\rho }\right) ^{1-\alpha }}%
,$ where $\tau \in \left( 0,x\right) ,$ we get%
\begin{eqnarray}
&&\frac{\rho ^{1-\beta }x^{k}}{\Gamma \left( \alpha \right) }\frac{\tau
^{\rho \left( \eta +1\right) -1}}{\left( x^{\rho }-\tau ^{\rho }\right)
^{1-\alpha }}h\left( \tau \right) H\left( \tau ,\gamma \right)  \notag \\
&=&\frac{\rho ^{1-\beta }x^{k}}{\Gamma \left( \alpha \right) }\frac{\tau
^{\rho \left( \eta +1\right) -1}}{\left( x^{\rho }-\tau ^{\rho }\right)
^{1-\alpha }}h\left( \tau \right) \varphi \left( \tau \right) \psi \left(
\tau \right) -\frac{\rho ^{1-\beta }x^{k}}{\Gamma \left( \alpha \right) }%
\frac{\tau ^{\rho \left( \eta +1\right) -1}}{\left( x^{\rho }-\tau ^{\rho
}\right) ^{1-\alpha }}h\left( \tau \right) \varphi \left( \tau \right) \psi
\left( \gamma \right)  \notag \\
&&-\frac{\rho ^{1-\beta }x^{k}}{\Gamma \left( \alpha \right) }\frac{\tau
^{\rho \left( \eta +1\right) -1}}{\left( x^{\rho }-\tau ^{\rho }\right)
^{1-\alpha }}h\left( \tau \right) \varphi \left( \gamma \right) \psi \left(
\tau \right) +\frac{\rho ^{1-\beta }x^{k}}{\Gamma \left( \alpha \right) }%
\frac{\tau ^{\rho \left( \eta +1\right) -1}}{\left( x^{\rho }-\tau ^{\rho
}\right) ^{1-\alpha }}h\left( \tau \right) \varphi \left( \gamma \right)
\psi \left( \gamma \right) .  \notag \\
&&  \label{INQ26}
\end{eqnarray}%
Now integrating (\ref{INQ26}) over $\left( 0,x\right) ,$ with respect to the
variable $\tau ,$ we obtain:%
\begin{eqnarray}
&&\frac{\rho ^{1-\beta }x^{k}}{\Gamma \left( \alpha \right) }\int_{0}^{x}%
\frac{\tau ^{\rho \left( \eta +1\right) -1}}{\left( x^{\rho }-\tau ^{\rho
}\right) ^{1-\alpha }}h\left( \tau \right) H\left( \tau ,\gamma \right) d\tau
\notag \\
&=&\text{ }^{\rho }\mathcal{I}_{\eta ,k}^{\alpha ,\beta }\left( h\varphi
\psi \right) \left( x\right) -\psi \left( \gamma \right) \text{ }^{\rho }%
\mathcal{I}_{\eta ,k}^{\alpha ,\beta }\left( h\varphi \right) \left( x\right)
\notag \\
&&-\varphi \left( \gamma \right) \text{ }^{\rho }\mathcal{I}_{\eta
,k}^{\alpha ,\beta }\left( h\psi \right) \left( x\right) +\varphi \left(
\gamma \right) \psi \left( \gamma \right) \text{ }^{\rho }\mathcal{I}_{\eta
,k}^{\alpha ,\beta }h\left( x\right) .  \label{INQ27}
\end{eqnarray}%
Now multiplying both sides of (\ref{INQ27}) by $\frac{\rho ^{1-\beta }x^{k}}{%
\Gamma \left( \alpha \right) }\frac{\gamma ^{\rho \left( \eta +1\right)
-1}h\left( \gamma \right) }{\left( x^{\rho }-\gamma ^{\rho }\right)
^{1-\alpha }},$ where $\gamma \in \left( 0,x\right) ,$and integrating the
resulting identity with respect to $\gamma $ from $0$ to $x$, we get%
\begin{eqnarray*}
&&\frac{\rho ^{2\left( 1-\beta \right) }x^{2k}}{\Gamma ^{2}\left( \alpha
\right) }\int_{0}^{x}\int_{0}^{x}\frac{\tau ^{\rho \left( \eta +1\right) -1}%
}{\left( x^{\rho }-\tau ^{\rho }\right) ^{1-\alpha }}\frac{\gamma ^{\rho
\left( \eta +1\right) -1}}{\left( x^{\rho }-\gamma ^{\rho }\right)
^{1-\alpha }}h\left( \tau \right) h\left( \gamma \right) H\left( \tau
,\gamma \right) d\tau d\gamma \\
&=&2\left[ \text{ }^{\rho }\mathcal{I}_{\eta ,k}^{\alpha ,\beta }\left(
h\varphi \psi \right) \left( x\right) \text{ }^{\rho }\mathcal{I}_{\eta
,k}^{\alpha ,\beta }h\left( x\right) -\text{ }^{\rho }\mathcal{I}_{\eta
,k}^{\alpha ,\beta }\left( h\psi \right) \left( x\right) \text{ }^{\rho }%
\mathcal{I}_{\eta ,k}^{\alpha ,\beta }\left( h\varphi \right) \left(
x\right) \right] .
\end{eqnarray*}%
Which is (\ref{inq24}) .
\end{proof}

\begin{theorem}
\label{thm4} Let $h$ be a positive function on $\left[ 0,\infty \right) ,$
and let $\varphi ,\psi ,$ be two differentiable functions on $\left[
0,\infty \right) .$ Suppose that $\varphi ^{\prime }\in L_{s}\left( \left[
0,\infty \right) \right) ,\psi ^{\prime }\in L_{v}\left( \left[ 0,\infty
\right) \right) ,s>1,\frac{1}{s}+\frac{1}{v}=1.$\ Then for all $x>0,\alpha
>0,\rho >0,k,\beta ,\eta \in 
\mathbb{R}
$ we have:%
\begin{eqnarray}
&&2\left\vert \text{ }^{\rho }\mathcal{I}_{\eta ,k}^{\alpha ,\beta }\left(
h\varphi \psi \right) \left( x\right) \text{ }^{\rho }\mathcal{I}_{\eta
,k}^{\alpha ,\beta }h\left( x\right) -\text{ }^{\rho }\mathcal{I}_{\eta
,k}^{\alpha ,\beta }\left( h\psi \right) \left( x\right) \text{ }^{\rho }%
\mathcal{I}_{\eta ,k}^{\alpha ,\beta }\left( h\varphi \right) \left(
x\right) \right\vert  \notag \\
&\leq &\frac{\left\Vert \varphi ^{\prime }\right\Vert _{s}\left\Vert \psi
^{\prime }\right\Vert _{v}\rho ^{2\left( 1-\beta \right) }x^{2k}}{\Gamma
^{2}\left( \alpha \right) }  \notag \\
&&\times \left[ \int_{0}^{x}\int_{0}^{x}\frac{\tau ^{\rho \left( \eta
+1\right) -1}}{\left( x^{\rho }-\tau ^{\rho }\right) ^{1-\alpha }}\frac{%
\gamma ^{\rho \left( \eta +1\right) -1}}{\left( x^{\rho }-\gamma ^{\rho
}\right) ^{1-\alpha }}h\left( \tau \right) h\left( \gamma \right) \left\vert
\tau -\gamma \right\vert d\tau d\gamma \right]  \notag \\
&\leq &\left\Vert \varphi ^{\prime }\right\Vert _{s}\left\Vert \psi ^{\prime
}\right\Vert _{v}x\left( \text{ }^{\rho }\mathcal{I}_{\eta ,k}^{\alpha
,\beta }h\left( x\right) \right) ^{2}.  \label{INQ33}
\end{eqnarray}
\end{theorem}

\begin{proof}
In lemma (\ref{lem3}),\ from the identity\ (\ref{INQ25}), we can write%
\begin{equation*}
H\left( \tau ,\gamma \right) =\int_{\tau }^{\gamma }\int_{\tau }^{\gamma
}\varphi ^{\prime }\left( t\right) \psi ^{\prime }\left( r\right) dtdr.
\end{equation*}%
By applying Holder inequality for double integral, we obtain:%
\begin{eqnarray}
\left\vert H\left( \tau ,\gamma \right) \right\vert &\leq &\left\vert
\int_{\tau }^{\gamma }\int_{\tau }^{\gamma }\left\vert \varphi ^{\prime
}\left( t\right) \right\vert ^{s}dtdr\right\vert ^{\frac{1}{s}}\left\vert
\int_{\tau }^{\gamma }\int_{\tau }^{\gamma }\left\vert \psi ^{\prime }\left(
r\right) \right\vert ^{v}dtdr\right\vert ^{\frac{1}{v}}  \notag \\
&\leq &\left( \left\vert \tau -\gamma \right\vert ^{\frac{1}{s}}\left\vert
\int_{\tau }^{\gamma }\left\vert \varphi ^{\prime }\left( t\right)
\right\vert ^{s}dt\right\vert ^{\frac{1}{s}}\right) \left( \left\vert \tau
-\gamma \right\vert ^{\frac{1}{v}}\left\vert \int_{\tau }^{\gamma
}\left\vert \psi ^{\prime }\left( r\right) \right\vert ^{v}dr\right\vert ^{%
\frac{1}{v}}\right)  \notag \\
&\leq &\left\vert \tau -\gamma \right\vert \left\vert \int_{\tau }^{\gamma
}\left\vert \varphi ^{\prime }\left( t\right) \right\vert ^{s}dt\right\vert
^{\frac{1}{s}}\left\vert \int_{\tau }^{\gamma }\left\vert \psi ^{\prime
}\left( r\right) \right\vert ^{v}dr\right\vert ^{\frac{1}{v}}.  \label{INQ28}
\end{eqnarray}%
Using inequality (\ref{INQ28}) in left-hand side of lemma (\ref{lem3}), we
can write%
\begin{eqnarray}
&&\frac{\rho ^{2\left( 1-\beta \right) }x^{2k}}{\Gamma ^{2}\left( \alpha
\right) }\int_{0}^{x}\int_{0}^{x}\frac{\tau ^{\rho \left( \eta +1\right) -1}%
}{\left( x^{\rho }-\tau ^{\rho }\right) ^{1-\alpha }}\frac{\gamma ^{\rho
\left( \eta +1\right) -1}}{\left( x^{\rho }-\gamma ^{\rho }\right)
^{1-\alpha }}h\left( \tau \right) h\left( \gamma \right) \left\vert H\left(
\tau ,\gamma \right) \right\vert d\tau d\gamma  \notag \\
&\leq &\frac{\rho ^{2\left( 1-\beta \right) }x^{2k}}{\Gamma ^{2}\left(
\alpha \right) }\int_{0}^{x}\int_{0}^{x}\frac{\tau ^{\rho \left( \eta
+1\right) -1}}{\left( x^{\rho }-\tau ^{\rho }\right) ^{1-\alpha }}\frac{%
\gamma ^{\rho \left( \eta +1\right) -1}}{\left( x^{\rho }-\gamma ^{\rho
}\right) ^{1-\alpha }}h\left( \tau \right) h\left( \gamma \right)  \notag \\
&&\times \left\vert \tau -\gamma \right\vert \left\vert \int_{\tau }^{\gamma
}\left\vert \varphi ^{\prime }\left( t\right) \right\vert ^{s}dt\right\vert
^{\frac{1}{s}}\left\vert \int_{\tau }^{\gamma }\left\vert \psi ^{\prime
}\left( r\right) \right\vert ^{v}dr\right\vert ^{\frac{1}{v}}d\tau d\gamma .
\label{INQ29}
\end{eqnarray}%
Again applying Holder inequality to the right-hand side of inequality (\ref%
{INQ29}), we obtain from (\ref{INQ29}) that%
\begin{eqnarray}
&&\frac{\rho ^{2\left( 1-\beta \right) }x^{2k}}{\Gamma ^{2}\left( \alpha
\right) }\int_{0}^{x}\int_{0}^{x}\frac{\tau ^{\rho \left( \eta +1\right) -1}%
}{\left( x^{\rho }-\tau ^{\rho }\right) ^{1-\alpha }}\frac{\gamma ^{\rho
\left( \eta +1\right) -1}}{\left( x^{\rho }-\gamma ^{\rho }\right)
^{1-\alpha }}h\left( \tau \right) h\left( \gamma \right) \left\vert H\left(
\tau ,\gamma \right) \right\vert d\tau d\gamma  \notag \\
&\leq &\left[ \frac{\rho ^{2\left( 1-\beta \right) }x^{2k}}{\Gamma
^{2}\left( \alpha \right) }\int_{0}^{x}\int_{0}^{x}\frac{\tau ^{\rho \left(
\eta +1\right) -1}}{\left( x^{\rho }-\tau ^{\rho }\right) ^{1-\alpha }}\frac{%
\gamma ^{\rho \left( \eta +1\right) -1}}{\left( x^{\rho }-\gamma ^{\rho
}\right) ^{1-\alpha }}h\left( \tau \right) h\left( \gamma \right) \left\vert
\tau -\gamma \right\vert \right.  \notag \\
&&\times \left. \left\vert \int_{\tau }^{\gamma }\left\vert \varphi ^{\prime
}\left( t\right) \right\vert ^{s}dt\right\vert d\tau d\gamma \right] ^{\frac{%
1}{s}}  \notag \\
&&\times \left[ \frac{\rho ^{2\left( 1-\beta \right) }x^{2k}}{\Gamma
^{2}\left( \alpha \right) }\int_{0}^{x}\int_{0}^{x}\frac{\tau ^{\rho \left(
\eta +1\right) -1}}{\left( x^{\rho }-\tau ^{\rho }\right) ^{1-\alpha }}\frac{%
\gamma ^{\rho \left( \eta +1\right) -1}}{\left( x^{\rho }-\gamma ^{\rho
}\right) ^{1-\alpha }}h\left( \tau \right) h\left( \gamma \right) \left\vert
\tau -\gamma \right\vert \right.  \notag \\
&&\left. \left\vert \int_{\tau }^{\gamma }\left\vert \psi ^{\prime }\left(
r\right) \right\vert ^{v}dr\right\vert d\tau d\gamma \right] ^{\frac{1}{v}}.
\label{INQ30}
\end{eqnarray}%
Since%
\begin{equation*}
\left\vert \int_{\tau }^{\gamma }\left\vert \varphi ^{\prime }\left(
t\right) \right\vert ^{s}dt\right\vert \leq \left\Vert \varphi ^{\prime
}\right\Vert _{s}^{s},
\end{equation*}%
\begin{equation}
\left\vert \int_{\tau }^{\gamma }\left\vert \psi ^{\prime }\left( r\right)
\right\vert ^{v}dr\right\vert \leq \left\Vert \psi ^{\prime }\right\Vert
_{v}^{v}.  \label{INQ38}
\end{equation}%
Then from (\ref{INQ30}), we have%
\begin{eqnarray*}
&&\frac{\rho ^{2\left( 1-\beta \right) }x^{2k}}{\Gamma ^{2}\left( \alpha
\right) }\int_{0}^{x}\int_{0}^{x}\frac{\tau ^{\rho \left( \eta +1\right) -1}%
}{\left( x^{\rho }-\tau ^{\rho }\right) ^{1-\alpha }}\frac{\gamma ^{\rho
\left( \eta +1\right) -1}}{\left( x^{\rho }-\gamma ^{\rho }\right)
^{1-\alpha }}h\left( \tau \right) h\left( \gamma \right) \left\vert H\left(
\tau ,\gamma \right) \right\vert d\tau d\gamma \\
&\leq &\left[ \frac{\left\Vert \varphi ^{\prime }\right\Vert _{s}^{s}\rho
^{2\left( 1-\beta \right) }x^{2k}}{\Gamma ^{2}\left( \alpha \right) }%
\int_{0}^{x}\int_{0}^{x}\frac{\tau ^{\rho \left( \eta +1\right) -1}}{\left(
x^{\rho }-\tau ^{\rho }\right) ^{1-\alpha }}\frac{\gamma ^{\rho \left( \eta
+1\right) -1}}{\left( x^{\rho }-\gamma ^{\rho }\right) ^{1-\alpha }}h\left(
\tau \right) h\left( \gamma \right) \left\vert \tau -\gamma \right\vert
d\tau d\gamma \right] ^{\frac{1}{s}} \\
&&\times \left[ \frac{\left\Vert \psi ^{\prime }\right\Vert _{v}^{v}\rho
^{2\left( 1-\beta \right) }x^{2k}}{\Gamma ^{2}\left( \alpha \right) }%
\int_{0}^{x}\int_{0}^{x}\frac{\tau ^{\rho \left( \eta +1\right) -1}}{\left(
x^{\rho }-\tau ^{\rho }\right) ^{1-\alpha }}\frac{\gamma ^{\rho \left( \eta
+1\right) -1}}{\left( x^{\rho }-\gamma ^{\rho }\right) ^{1-\alpha }}h\left(
\tau \right) h\left( \gamma \right) \left\vert \tau -\gamma \right\vert
d\tau d\gamma \right] ^{\frac{1}{v}}.
\end{eqnarray*}%
So we have 
\begin{eqnarray*}
&&\frac{\rho ^{2\left( 1-\beta \right) }x^{2k}}{\Gamma ^{2}\left( \alpha
\right) }\int_{0}^{x}\int_{0}^{x}\frac{\tau ^{\rho \left( \eta +1\right) -1}%
}{\left( x^{\rho }-\tau ^{\rho }\right) ^{1-\alpha }}\frac{\gamma ^{\rho
\left( \eta +1\right) -1}}{\left( x^{\rho }-\gamma ^{\rho }\right)
^{1-\alpha }}h\left( \tau \right) h\left( \gamma \right) \left\vert H\left(
\tau ,\gamma \right) \right\vert d\tau d\gamma \\
&\leq &\frac{\left\Vert \varphi ^{\prime }\right\Vert _{s}\left\Vert \psi
^{\prime }\right\Vert _{v}\rho ^{2\left( 1-\beta \right) }x^{2k}}{\Gamma
^{2}\left( \alpha \right) } \\
&&\times \left[ \int_{0}^{x}\int_{0}^{x}\frac{\tau ^{\rho \left( \eta
+1\right) -1}}{\left( x^{\rho }-\tau ^{\rho }\right) ^{1-\alpha }}\frac{%
\gamma ^{\rho \left( \eta +1\right) -1}}{\left( x^{\rho }-\gamma ^{\rho
}\right) ^{1-\alpha }}h\left( \tau \right) h\left( \gamma \right) \left\vert
\tau -\gamma \right\vert d\tau d\gamma \right] ^{\frac{1}{s}} \\
&&\times \left[ \int_{0}^{x}\int_{0}^{x}\frac{\tau ^{\rho \left( \eta
+1\right) -1}}{\left( x^{\rho }-\tau ^{\rho }\right) ^{1-\alpha }}\frac{%
\gamma ^{\rho \left( \eta +1\right) -1}}{\left( x^{\rho }-\gamma ^{\rho
}\right) ^{1-\alpha }}h\left( \tau \right) h\left( \gamma \right) \left\vert
\tau -\gamma \right\vert d\tau d\gamma \right] ^{\frac{1}{v}}.
\end{eqnarray*}%
Hence%
\begin{eqnarray}
&&\frac{\rho ^{2\left( 1-\beta \right) }x^{2k}}{\Gamma ^{2}\left( \alpha
\right) }\int_{0}^{x}\int_{0}^{x}\frac{\tau ^{\rho \left( \eta +1\right) -1}%
}{\left( x^{\rho }-\tau ^{\rho }\right) ^{1-\alpha }}\frac{\gamma ^{\rho
\left( \eta +1\right) -1}}{\left( x^{\rho }-\gamma ^{\rho }\right)
^{1-\alpha }}h\left( \tau \right) h\left( \gamma \right) \left\vert H\left(
\tau ,\gamma \right) \right\vert d\tau d\gamma  \notag \\
&\leq &\frac{\left\Vert \varphi ^{\prime }\right\Vert _{s}\left\Vert \psi
^{\prime }\right\Vert _{v}\rho ^{2\left( 1-\beta \right) }x^{2k}}{\Gamma
^{2}\left( \alpha \right) }  \notag \\
&&\times \left[ \int_{0}^{x}\int_{0}^{x}\frac{\tau ^{\rho \left( \eta
+1\right) -1}}{\left( x^{\rho }-\tau ^{\rho }\right) ^{1-\alpha }}\frac{%
\gamma ^{\rho \left( \eta +1\right) -1}}{\left( x^{\rho }-\gamma ^{\rho
}\right) ^{1-\alpha }}h\left( \tau \right) h\left( \gamma \right) \left\vert
\tau -\gamma \right\vert d\tau d\gamma \right] .  \label{INQ31}
\end{eqnarray}%
Using lemma (\ref{lem3}), and the inequality (\ref{INQ31}), with the
properties of the modulus, we get%
\begin{eqnarray}
&&2\left\vert \text{ }^{\rho }\mathcal{I}_{\eta ,k}^{\alpha ,\beta }\left(
h\varphi \psi \right) \left( x\right) \text{ }^{\rho }\mathcal{I}_{\eta
,k}^{\alpha ,\beta }h\left( x\right) -\text{ }^{\rho }\mathcal{I}_{\eta
,k}^{\alpha ,\beta }\left( h\psi \right) \left( x\right) \text{ }^{\rho }%
\mathcal{I}_{\eta ,k}^{\alpha ,\beta }\left( h\varphi \right) \left(
x\right) \right\vert  \notag \\
&\leq &\frac{\left\Vert \varphi ^{\prime }\right\Vert _{s}\left\Vert \psi
^{\prime }\right\Vert _{v}\rho ^{2\left( 1-\beta \right) }x^{2k}}{\Gamma
^{2}\left( \alpha \right) }  \notag \\
&&\times \left[ \int_{0}^{x}\int_{0}^{x}\frac{\tau ^{\rho \left( \eta
+1\right) -1}}{\left( x^{\rho }-\tau ^{\rho }\right) ^{1-\alpha }}\frac{%
\gamma ^{\rho \left( \eta +1\right) -1}}{\left( x^{\rho }-\gamma ^{\rho
}\right) ^{1-\alpha }}h\left( \tau \right) h\left( \gamma \right) \left\vert
\tau -\gamma \right\vert d\tau d\gamma \right] .  \label{INQ32}
\end{eqnarray}%
Which proves frist part of (\ref{INQ33}). To prove the second inequality of (%
\ref{INQ33}), we have%
\begin{equation*}
0\leq \tau \leq x,0\leq \gamma \leq x.
\end{equation*}%
Then 
\begin{equation*}
0\leq \left\vert \tau -\gamma \right\vert \leq x.
\end{equation*}%
Hence%
\begin{eqnarray}
&&\frac{\rho ^{2\left( 1-\beta \right) }x^{2k}}{\Gamma ^{2}\left( \alpha
\right) }\int_{0}^{x}\int_{0}^{x}\frac{\tau ^{\rho \left( \eta +1\right) -1}%
}{\left( x^{\rho }-\tau ^{\rho }\right) ^{1-\alpha }}\frac{\gamma ^{\rho
\left( \eta +1\right) -1}}{\left( x^{\rho }-\gamma ^{\rho }\right)
^{1-\alpha }}h\left( \tau \right) h\left( \gamma \right) \left\vert H\left(
\tau ,\gamma \right) \right\vert d\tau d\gamma  \notag \\
&\leq &\frac{\left\Vert \varphi ^{\prime }\right\Vert _{s}\left\Vert \psi
^{\prime }\right\Vert _{v}\rho ^{2\left( 1-\beta \right) }x^{1+2k}}{\Gamma
^{2}\left( \alpha \right) }  \notag \\
&&\times \left[ \int_{0}^{x}\int_{0}^{x}\frac{\tau ^{\rho \left( \eta
+1\right) -1}}{\left( x^{\rho }-\tau ^{\rho }\right) ^{1-\alpha }}\frac{%
\gamma ^{\rho \left( \eta +1\right) -1}}{\left( x^{\rho }-\gamma ^{\rho
}\right) ^{1-\alpha }}h\left( \tau \right) h\left( \gamma \right) d\tau
d\gamma \right] .  \notag \\
&=&\left\Vert \varphi ^{\prime }\right\Vert _{s}\left\Vert \psi ^{\prime
}\right\Vert _{v}x\left( \text{ }^{\rho }\mathcal{I}_{\eta ,k}^{\alpha
,\beta }h\left( x\right) \right) ^{2}.  \label{INQ34}
\end{eqnarray}%
Which proves second inequality of (\ref{INQ33}). Hence, Theorem (\ref{thm4})
is proved.
\end{proof}

Now we give the following theorem with different parameters:

\begin{theorem}
\label{thm6} Let $h$ be a positive function on $\left[ 0,\infty \right) ,$
and let $\varphi ,\psi ,$ be two differentiable functions on $\left[
0,\infty \right) .$ suppose that $\varphi ^{\prime }\in L_{s}\left( \left[
0,\infty \right) \right) ,\psi ^{\prime }\in L_{v}\left( \left[ 0,\infty
\right) \right) ,s>1,\frac{1}{s}+\frac{1}{v}=1.$\ Then for all $x>0,\alpha
>0,\delta >0,\rho >0,k,\beta ,\lambda ,\eta \in 
\mathbb{R}
,$ we have:%
\begin{eqnarray}
&&^{\rho }\mathcal{I}_{\eta ,k}^{\alpha ,\beta }\left( h\varphi \psi \right)
\left( x\right) \text{ }^{\rho }\mathcal{I}_{\eta ,k}^{\delta ,\lambda
}h\left( x\right) -\text{ }^{\rho }\mathcal{I}_{\eta ,k}^{\delta ,\lambda
}\left( h\psi \right) \left( x\right) \text{ }^{\rho }\mathcal{I}_{\eta
,k}^{\alpha ,\beta }\left( h\varphi \right) \left( x\right)   \notag \\
&&\text{ }^{\rho }\mathcal{I}_{\eta ,k}^{\delta ,\lambda }\left( h\varphi
\right) \left( x\right) \text{ }^{\rho }\mathcal{I}_{\eta ,k}^{\alpha ,\beta
}\left( h\psi \right) \left( x\right) +\text{ }^{\rho }\mathcal{I}_{\eta
,k}^{\delta ,\lambda }\left( h\varphi \psi \right) \left( x\right) \text{ }%
^{\rho }\mathcal{I}_{\eta ,k}^{\alpha ,\beta }h\left( x\right)   \notag \\
&\leq &\frac{\left\Vert \varphi ^{\prime }\right\Vert _{s}\left\Vert \psi
^{\prime }\right\Vert _{v}\rho ^{2-\left( \beta -\lambda \right) }x^{2k}}{%
\Gamma \left( \alpha \right) \Gamma \left( \delta \right) }  \notag \\
&&\times \left[ \int_{0}^{x}\int_{0}^{x}\frac{\tau ^{\rho \left( \eta
+1\right) -1}}{\left( x^{\rho }-\tau ^{\rho }\right) ^{1-\alpha }}\frac{%
\gamma ^{\rho \left( \eta +1\right) -1}}{\left( x^{\rho }-\gamma ^{\rho
}\right) ^{1-\delta }}h\left( \tau \right) h\left( \gamma \right) \left\vert
\tau -\gamma \right\vert d\tau d\gamma \right]   \notag \\
&\leq &\left\Vert \varphi ^{\prime }\right\Vert _{s}\left\Vert \psi ^{\prime
}\right\Vert _{v}x\left( \text{ }^{\rho }\mathcal{I}_{\eta ,k}^{\alpha
,\beta }h\left( x\right) \text{ }^{\rho }\mathcal{I}_{\eta ,k}^{\delta
,\lambda }h\left( x\right) \right) .  \label{INQ40}
\end{eqnarray}
\end{theorem}

\begin{proof}
In lemma (\ref{lem3}), multiplying both sides of (\ref{INQ27}) by $\frac{%
\rho ^{1-\lambda }x^{k}}{\Gamma \left( \delta \right) }\frac{\gamma ^{\rho
\left( \eta +1\right) -1}h\left( \gamma \right) }{\left( x^{\rho }-\gamma
^{\rho }\right) ^{1-\delta }},$ where $\gamma \in \left( 0,x\right) ,$and
integrating the resulting identity with respect to $\gamma $ from $0$ to $x$%
, we can write%
\begin{eqnarray}
&&\frac{\rho ^{2-\left( \beta -\lambda \right) }x^{2k}}{\Gamma \left( \alpha
\right) \Gamma \left( \delta \right) }\int_{0}^{x}\int_{0}^{x}\frac{\tau
^{\rho \left( \eta +1\right) -1}}{\left( x^{\rho }-\tau ^{\rho }\right)
^{1-\alpha }}\frac{\gamma ^{\rho \left( \eta +1\right) -1}}{\left( x^{\rho
}-\gamma ^{\rho }\right) ^{1-\delta }}h\left( \gamma \right) h\left( \tau
\right) H\left( \tau ,\gamma \right) d\tau   \notag \\
&=&\text{ }^{\rho }\mathcal{I}_{\eta ,k}^{\alpha ,\beta }\left( h\varphi
\psi \right) \left( x\right) \text{ }^{\rho }\mathcal{I}_{\eta ,k}^{\delta
,\lambda }h\left( x\right) -\text{ }^{\rho }\mathcal{I}_{\eta ,k}^{\delta
,\lambda }\left( h\psi \right) \left( x\right) \text{ }^{\rho }\mathcal{I}%
_{\eta ,k}^{\alpha ,\beta }\left( h\varphi \right) \left( x\right)   \notag
\\
&&-\text{ }^{\rho }\mathcal{I}_{\eta ,k}^{\delta ,\lambda }\left( h\varphi
\right) \left( x\right) \text{ }^{\rho }\mathcal{I}_{\eta ,k}^{\alpha ,\beta
}\left( h\psi \right) \left( x\right) +\text{ }^{\rho }\mathcal{I}_{\eta
,k}^{\delta ,\lambda }\left( h\varphi \psi \right) \left( x\right) \text{ }%
^{\rho }\mathcal{I}_{\eta ,k}^{\alpha ,\beta }h\left( x\right) .
\label{INQ35}
\end{eqnarray}%
Using inequality (\ref{INQ28}) in left-hand side of (\ref{INQ35}), we can
write%
\begin{eqnarray}
&&\frac{\rho ^{2-\left( \beta -\lambda \right) }x^{2k}}{\Gamma \left( \alpha
\right) \Gamma \left( \delta \right) }\int_{0}^{x}\int_{0}^{x}\frac{\tau
^{\rho \left( \eta +1\right) -1}}{\left( x^{\rho }-\tau ^{\rho }\right)
^{1-\alpha }}\frac{\gamma ^{\rho \left( \eta +1\right) -1}}{\left( x^{\rho
}-\gamma ^{\rho }\right) ^{1-\delta }}h\left( \gamma \right) h\left( \tau
\right) \left\vert H\left( \tau ,\gamma \right) \right\vert d\tau   \notag \\
&\leq &\frac{\rho ^{2-\left( \beta -\lambda \right) }x^{2k}}{\Gamma \left(
\alpha \right) \Gamma \left( \delta \right) }\int_{0}^{x}\int_{0}^{x}\frac{%
\tau ^{\rho \left( \eta +1\right) -1}}{\left( x^{\rho }-\tau ^{\rho }\right)
^{1-\alpha }}\frac{\gamma ^{\rho \left( \eta +1\right) -1}}{\left( x^{\rho
}-\gamma ^{\rho }\right) ^{1-\delta }}h\left( \gamma \right) h\left( \tau
\right)   \notag \\
&&\times \left\vert \tau -\gamma \right\vert \left\vert \int_{\tau }^{\gamma
}\left\vert \varphi ^{\prime }\left( t\right) \right\vert ^{s}dt\right\vert
^{\frac{1}{s}}\left\vert \int_{\tau }^{\gamma }\left\vert \psi ^{\prime
}\left( r\right) \right\vert ^{v}dr\right\vert ^{\frac{1}{v}}d\tau d\gamma .
\label{INQ36}
\end{eqnarray}%
By applying Holder inequality for double integral to above inequality, we
obtain:%
\begin{eqnarray}
&&\frac{\rho ^{2-\left( \beta -\lambda \right) }x^{2k}}{\Gamma \left( \alpha
\right) \Gamma \left( \delta \right) }\int_{0}^{x}\int_{0}^{x}\frac{\tau
^{\rho \left( \eta +1\right) -1}}{\left( x^{\rho }-\tau ^{\rho }\right)
^{1-\alpha }}\frac{\gamma ^{\rho \left( \eta +1\right) -1}}{\left( x^{\rho
}-\gamma ^{\rho }\right) ^{1-\delta }}h\left( \tau \right) h\left( \gamma
\right) \left\vert H\left( \tau ,\gamma \right) \right\vert d\tau d\gamma  
\notag \\
&\leq &\left[ \frac{\rho ^{2-\left( \beta -\lambda \right) }x^{2k}}{\Gamma
\left( \alpha \right) \Gamma \left( \delta \right) }\int_{0}^{x}\int_{0}^{x}%
\frac{\tau ^{\rho \left( \eta +1\right) -1}}{\left( x^{\rho }-\tau ^{\rho
}\right) ^{1-\alpha }}\frac{\gamma ^{\rho \left( \eta +1\right) -1}}{\left(
x^{\rho }-\gamma ^{\rho }\right) ^{1-\delta }}h\left( \tau \right) h\left(
\gamma \right) \left\vert \tau -\gamma \right\vert \right.   \notag \\
&&\times \left. \left\vert \int_{\tau }^{\gamma }\left\vert \varphi ^{\prime
}\left( t\right) \right\vert ^{s}dt\right\vert d\tau d\gamma \right] ^{\frac{%
1}{s}}  \notag \\
&&\times \left[ \frac{\rho ^{2-\left( \beta -\lambda \right) }x^{2k}}{\Gamma
\left( \alpha \right) \Gamma \left( \delta \right) }\int_{0}^{x}\int_{0}^{x}%
\frac{\tau ^{\rho \left( \eta +1\right) -1}}{\left( x^{\rho }-\tau ^{\rho
}\right) ^{1-\alpha }}\frac{\gamma ^{\rho \left( \eta +1\right) -1}}{\left(
x^{\rho }-\gamma ^{\rho }\right) ^{1-\delta }}h\left( \tau \right) h\left(
\gamma \right) \left\vert \tau -\gamma \right\vert \right.   \notag \\
&&\left. \left\vert \int_{\tau }^{\gamma }\left\vert \psi ^{\prime }\left(
r\right) \right\vert ^{v}dr\right\vert d\tau d\gamma \right] ^{\frac{1}{v}}.
\label{INQ37}
\end{eqnarray}%
Using (\ref{INQ38}) and (\ref{INQ37}), we can write%
\begin{eqnarray}
&&\frac{\rho ^{2-\left( \beta -\lambda \right) }x^{2k}}{\Gamma \left( \alpha
\right) \Gamma \left( \delta \right) }\int_{0}^{x}\int_{0}^{x}\frac{\tau
^{\rho \left( \eta +1\right) -1}}{\left( x^{\rho }-\tau ^{\rho }\right)
^{1-\alpha }}\frac{\gamma ^{\rho \left( \eta +1\right) -1}}{\left( x^{\rho
}-\gamma ^{\rho }\right) ^{1-\delta }}h\left( \tau \right) h\left( \gamma
\right) \left\vert H\left( \tau ,\gamma \right) \right\vert d\tau d\gamma  
\notag \\
&\leq &\frac{\left\Vert \varphi ^{\prime }\right\Vert _{s}\left\Vert \psi
^{\prime }\right\Vert _{v}\rho ^{2-\left( \beta -\lambda \right) }x^{2k}}{%
\Gamma \left( \alpha \right) \Gamma \left( \delta \right) }  \notag \\
&&\times \left[ \int_{0}^{x}\int_{0}^{x}\frac{\tau ^{\rho \left( \eta
+1\right) -1}}{\left( x^{\rho }-\tau ^{\rho }\right) ^{1-\alpha }}\frac{%
\gamma ^{\rho \left( \eta +1\right) -1}}{\left( x^{\rho }-\gamma ^{\rho
}\right) ^{1-\delta }}h\left( \tau \right) h\left( \gamma \right) \left\vert
\tau -\gamma \right\vert d\tau d\gamma \right] .  \label{INQ39}
\end{eqnarray}%
Using (\ref{INQ35}) and (\ref{INQ39}), with the properties of the modulus,
we get%
\begin{eqnarray}
&&^{\rho }\mathcal{I}_{\eta ,k}^{\alpha ,\beta }\left( h\varphi \psi \right)
\left( x\right) \text{ }^{\rho }\mathcal{I}_{\eta ,k}^{\delta ,\lambda
}h\left( x\right) -\text{ }^{\rho }\mathcal{I}_{\eta ,k}^{\delta ,\lambda
}\left( h\psi \right) \left( x\right) \text{ }^{\rho }\mathcal{I}_{\eta
,k}^{\alpha ,\beta }\left( h\varphi \right) \left( x\right)   \notag \\
&&\text{ }^{\rho }\mathcal{I}_{\eta ,k}^{\delta ,\lambda }\left( h\varphi
\right) \left( x\right) \text{ }^{\rho }\mathcal{I}_{\eta ,k}^{\alpha ,\beta
}\left( h\psi \right) \left( x\right) +\text{ }^{\rho }\mathcal{I}_{\eta
,k}^{\delta ,\lambda }\left( h\varphi \psi \right) \left( x\right) \text{ }%
^{\rho }\mathcal{I}_{\eta ,k}^{\alpha ,\beta }h\left( x\right)   \notag \\
&\leq &\frac{\left\Vert \varphi ^{\prime }\right\Vert _{s}\left\Vert \psi
^{\prime }\right\Vert _{v}\rho ^{2-\left( \beta -\lambda \right) }x^{2k}}{%
\Gamma \left( \alpha \right) \Gamma \left( \delta \right) }  \notag \\
&&\times \left[ \int_{0}^{x}\int_{0}^{x}\frac{\tau ^{\rho \left( \eta
+1\right) -1}}{\left( x^{\rho }-\tau ^{\rho }\right) ^{1-\alpha }}\frac{%
\gamma ^{\rho \left( \eta +1\right) -1}}{\left( x^{\rho }-\gamma ^{\rho
}\right) ^{1-\delta }}h\left( \tau \right) h\left( \gamma \right) \left\vert
\tau -\gamma \right\vert d\tau d\gamma \right] .  \label{INQ41}
\end{eqnarray}%
Which proves first inequality of (\ref{INQ4}). Second inequality can be
proved similarly.
\end{proof}

\begin{remark}
If $k=0,\eta =0$ and taking $\rho \rightarrow 1$ in the theorems (\ref{thm4}%
) and (\ref{thm6}), we get  
\begin{eqnarray*}
&&\emph{2}\left\vert \mathcal{I}^{\alpha }h\left( x\right) \mathcal{I}%
^{\alpha }\left( h\varphi \psi \right) \left( x\right) -\mathcal{I}^{\alpha
}\left( h\varphi \right) \left( x\right) \mathcal{I}^{\alpha }\left( h\psi
\right) \left( x\right) \right\vert  \\
&\leq &\frac{\left\Vert \varphi ^{\prime }\right\Vert _{s}\left\Vert \psi
^{\prime }\right\Vert _{v}}{\Gamma ^{2}\left( \alpha \right) }%
\int_{a}^{b}\int_{a}^{b}\left( x-\tau \right) ^{\alpha -1}\left( x-\gamma
\right) ^{\alpha -1}\left\vert \tau -\gamma \right\vert h\left( \tau \right)
h\left( \gamma \right) d\tau d\gamma  \\
&\leq &\left\Vert \varphi ^{\prime }\right\Vert _{s}\left\Vert \psi ^{\prime
}\right\Vert _{v}x\left( \mathcal{I}^{\alpha }h\left( x\right) \right) ^{2}
\end{eqnarray*}%
and%
\begin{eqnarray*}
&&\mathcal{I}^{\alpha }\left( h\varphi \psi \right) \left( x\right) \mathcal{%
I}^{\delta }h\left( x\right) -\mathcal{I}^{\delta }\left( h\psi \right)
\left( x\right) \mathcal{I}^{\alpha }\left( h\varphi \right) \left( x\right) 
\\
&&\mathcal{I}^{\delta }\left( h\varphi \right) \left( x\right) \mathcal{I}%
^{\alpha }\left( h\psi \right) \left( x\right) +\mathcal{I}^{\delta }\left(
h\varphi \psi \right) \left( x\right) \mathcal{I}^{\alpha }h\left( x\right) 
\\
&\leq &\frac{\left\Vert \varphi ^{\prime }\right\Vert _{s}\left\Vert \psi
^{\prime }\right\Vert _{v}}{\Gamma \left( \alpha \right) \Gamma \left(
\delta \right) }\int_{0}^{x}\int_{0}^{x}\left( x-\tau \right) ^{\alpha
-1}\left( x-\gamma \right) ^{\delta -1}h\left( \tau \right) h\left( \gamma
\right) \left\vert \tau -\gamma \right\vert d\tau d\gamma  \\
&&\leq \left\Vert \varphi ^{\prime }\right\Vert _{s}\left\Vert \psi ^{\prime
}\right\Vert _{v}x\left( \mathcal{I}^{\alpha }h\left( x\right) \mathcal{I}%
^{\delta }h\left( x\right) \right) .
\end{eqnarray*}%
respectively, as in (see \cite{rr}). Similarly we can get the remaining five
cases of generalized fractional integral mentioned at the preliminaries.
\end{remark}

\bigskip

Department of Mathematics Dr. Babasaheb Ambedkar Marathwada University
Aurangabad-431 004 India.

E-mail: tariq10011@gmail.com.

\bigskip

Department of Mathematics Dr. Babasaheb Ambedkar Marathwada University
Aurangabad-431 004 India.

E-mail: pachpatte@gmail.com.

\end{document}